\documentclass[11pt]{article}
\usepackage[margin=2.5cm]{geometry}
\usepackage{pdflscape}
\usepackage{amsfonts}
\usepackage{amsmath}
\usepackage{amsthm}
\usepackage{amssymb}
\usepackage{amsmath,amscd}
\usepackage{graphicx}
\usepackage{enumerate}
\usepackage{subcaption}
\graphicspath{ {./Figure/} }
\usepackage[pdffitwindow=false,
colorlinks=true,linkcolor=black,urlcolor=black,citecolor=black]{hyperref}

\newtheorem{theorem}{Theorem}[section]
\newtheorem{lemma}[theorem]{Lemma}
\theoremstyle{definition}
\newtheorem{definition}[theorem]{Definition}

\newtheorem{proposition}[theorem]{Proposition}

\newtheorem{remark}[theorem]{Remark}
\numberwithin{equation}{section}

\newcommand{\id}{\mathrm{Id}}

\newcommand{\spin}{\mathrm{Spin}}

\newcommand{\trace}{\mathrm{tr}}

\newcommand{\hmo}{\mathrm{Hom}}

\newcommand{\ricci}{\mathrm{Ric}}

\begin{document}
\title {Einstein metrics of cohomogeneity one with $\mathbb{S}^{4m+3}$ as principal orbit}
\author{Hanci Chi}
\date{\today}
\maketitle

\abstract
In this article, we construct non-compact complete Einstein metrics on two infinite series of manifolds. The first series of manifolds are vector bundles with $\mathbb{S}^{4m+3}$ as principal orbit and $\mathbb{HP}^{m}$ as singular orbit. The second series of manifolds are $\mathbb{R}^{4m+4}$ with the same principal orbit. For each case, a continuous 1-parameter family of complete Ricci-flat metrics and a continuous 2-parameter family of complete negative Einstein metrics are constructed. In particular, $\spin(7)$ metrics $\mathbb{A}_8$ and $\mathbb{B}_8$ discovered by Cveti\v{c} et al. in 2004 are recovered in the Ricci-flat family. A Ricci flat metric with conical singularity is also constructed on $\mathbb{R}^{4m+4}$. Asymptotic limits of all Einstein metrics constructed are studied. Most of the Ricci-flat metrics are asymptotically locally conical (ALC). Asymptotically conical (AC) metrics are found on the boundary of the Ricci-flat family. All the negative Einstein metrics constructed are asymptotically hyperbolic (AH).

\tableofcontents
\section{Introduction}
A Riemannian manifold $(M, g)$ is \emph{Einstein} if its Ricci curvature satisfies $\ricci(g)=\Lambda g$ for some constant $\Lambda$. A Riemannian manifold $(M,g)$ is of cohomogeneity one if a Lie Group $G$ acts isometrically on $M$ with principal orbit $G/K$ of codimension one. The Einstein equations of a cohomogeneity one manifold is reduced to a dynamic system.

In this article we focus on constructing non-compact cohomogeneity one Einstein metrics. Known examples include the first inhomogeneous Einstein metric in \cite{calabi_construction_1975}, which has K\"ahler holonomy. More non-compact K\"ahler--Einstein metrics of cohomogeneity one were constructed in \cite{berard-bergery_sur_1982}\cite{dancer_kahler-einstein_1998}\cite{wang_einstein_1998}\cite{dancer_einstein_2002}. Non-compact cohomogeneity one $G_2$ and $\spin(7)$ metrics, which are motivations to this article, were constructed in \cite{bryant_construction_1989}\cite{gibbons_einstein_1990}\cite{cvetic_new_2004}\cite{foscolo_infinitely_2018}. Fixing the principal orbit $G/K=Sp(m+1)U(1)/Sp(m)\Delta U(1)$, we aim to look into the full dynamic system of cohomogeneity one Einstein metrics without imposing any special holonomy condition. Odd dimensional cohomogeneity one Einstein metrics with generic holonomy include those constructed in \cite{berard-bergery_sur_1982}\cite{wang_einstein_1998}\cite{chen_examples_2011}. The case where the isotropy representation of the principal orbit consists of exactly two inequivalent irreducible summands was studied in \cite{bohm_non-compact_1999}\cite{wink_cohomogeneity_2017}. Examples where the principal orbit is a product of irreducible homogeneous spaces was constructed in \cite{bohm_non-compact_1999}. In \cite{chi_invariant_2019}, Ricci-flat metrics with Wallach spaces as principal orbits were constructed. The isotropy representation of Wallach spaces consists of three inequivalent irreducible summands, two of which are from the singular orbit, allowing the singular orbit to be squashed. In this article, the principal orbit also consists of three irreducible summands. Our main results are the following.
\begin{theorem}
\label{zeta curve}
Let $M$ be the $\mathbb{R}^4$-bundle over $\mathbb{HP}^m$ given by the group triple $(G,H,K)=(Sp(m+1)U(1),Sp(m)Sp(1)U(1),Sp(m)\Delta U(1))$.
There exists a continuous $2$-parameter family of smooth Einstein metrics $\{\zeta_{(s_1,s_2,s_3)}\mid (s_1,s_2,s_3)\in\mathbb{S}^2, s_1>0,s_2,s_3\geq 0\}$ of cohomogeneity one  on $M$. Specifically,
\begin{enumerate}
\item
$\zeta_{(s_1,s_2,0)}$ is a continuous 1-parameter family of complete Ricci-flat metrics on $M$. A metric in this family is AC if $s_2=0$, it is ALC otherwise.  For $m=1$, each $\zeta_{(s_1,s_2,0)}$ has holonomy $\spin(7)$ on $M^8$. For $m>1$, each $\zeta_{(s_1,s_2,0)}$ with $s_2>0$ has generic holonomy.
\item 
$\zeta_{(s_1,s_2,s_3)}$ with $s_3>0$ is a continuous 2-parameter family of complete AH negative Einstein metrics on $M$.
\end{enumerate}
\end{theorem}

Some known Einstein metrics are recovered in this family. In the case where $m=1$, $\zeta_{(1,0,0)}$ is the $\spin(7)$ metric in \cite{bryant_construction_1989}\cite{gibbons_einstein_1990}. The 1-parameter family of $\spin(7)$ metrics $\zeta_{(s_1,s_2,0)}$ was constructed in \cite{cvetic_new_2004}. For all $m\geq 1$, metrics $\zeta_{(s_1,0,s_3)}$ are of two summands type. They were constructed in \cite{bohm_non-compact_1999}\cite{wink_cohomogeneity_2017}. All the other metrics in $\zeta_{(s_1,s_2,s_3)}$ are new to the author.

On $\mathbb{R}^{4m+4}$, we have the following.
\begin{theorem}
\label{gamma curve}
There exists a continuous $2$-parameter family of smooth Einstein metrics $\{\gamma_{(s_1,s_2,s_3)}\mid (s_1,s_2,s_3)\in\mathbb{S}^2, s_1,s_2,s_3\geq 0 \}$ of cohomogeneity one  on $\mathbb{R}^{4m+4}$. Specifically,
\begin{enumerate}
\item
$\gamma_{(s_1,s_2,0)}$ is a continuous 1-parameter family of complete Ricci-flat metric on $\mathbb{R}^{4m+4}$. A metric in this family is AC if $s_2=0$, it is ALC otherwise. For $m=1$, $\gamma_{\left(\frac{1}{\sqrt{5}},\frac{2}{\sqrt{5}},0\right)}$ is $\spin(7)$ on $\mathbb{R}^8$ and all the other Ricci-flat metrics have generic holonomy. For $m>1$, each $\gamma_{(s_1,s_2,0)}$ with $s_2>0$ has generic holonomy.
\item 
$\gamma_{(s_1,s_2,s_3)}$ with $s_3>0$ is a continuous 2-parameter family of complete AH negative Einstein metric on $\mathbb{R}^{4m+4}$. In particular, $\gamma_{(0,0,1)}$ is the hyperbolic cone with base the standard $\mathbb{S}^{4m+3}$.
\end{enumerate}
\end{theorem}
Although not included in the theorem above, the parameter $(s_1,s_2,s_3)$ can be the origin for $\gamma_{(s_1,s_2,s_3)}$. The metric represented is the Euclidean metric on $\mathbb{R}^{4m+4}$, as shown in Section \ref{Critical Points}. 
Metrics $\gamma_{(0,s_2,s_3)}$ are of two summand type. They first appeared in \cite{berard-bergery_sur_1982}. Metrics $\gamma_{(s_1,0,s_3)}$ is also of two summands type. They were constructed in \cite{chi_cohomogeneity_2019}. In the case where $m=1$, $\gamma_{\left(\frac{1}{\sqrt{5}},\frac{2}{\sqrt{5}},0\right)}$ is the $\spin(7)$ metric with the opposite chirality to the metric $\mathbb{A}_8$ constructed in \cite{cvetic_new_2004}. All the other metrics in $\gamma_{(s_1,s_2,s_3)}$ are new to the author. 

In some sense, the 2-dimensional parameter $(s_1,s_2,s_3)\in\mathbb{S}^2$ in Theorem \ref{zeta curve} and Theorem \ref{gamma curve} controls the asymptotic limit of the metric represented. The non-vanishing of $s_2$ in $(s_1,s_2,0)$ gives the ALC asymptotics. The parameter also describes how the principal orbit is squashed near the singular orbit. More details are discussed in Section \ref{Critical Points}. The non-vanishing of $s_3$ gives the AH asymptotics. As discussed in Section \ref{Einstein Equation}, the dynamic system of the negative Einstein metrics has a subsystem that can represent the Ricci-flat system. Integral curves with $s_3=0$ are solutions of this subsystem. 

New Taub-NUT metrics on $\mathbb{R}^{4m+4}$ with conical singularity at the origin are also constructed.
\begin{theorem}
\label{Gamma curve}
There exists a continuous $1$-parameter family of Einstein metrics $\{\Gamma_{s}\mid s\in[0,\epsilon)\}$ of cohomogeneity one on $\mathbb{R}^{4m+4}$. They all have conical singularity at the origin. Specifically,
\begin{enumerate}
\item
$\Gamma_{0}$ a singular ALC Ricci-flat metric on $\mathbb{R}^{4m+4}$. For $m=1$, the metric is $\spin(7)$ on $\mathbb{R}^8$. For $m>1$, the metric has generic holonomy.
\item 
$\Gamma_{s}$ with $s>0$ is a continuous 1-parameter family of singular AH negative Einstein metric on $\mathbb{R}^{4m+4}$.
\end{enumerate}
\end{theorem}

Consider the holonomy of the Ricci-flat metrics in Theorems \ref{zeta curve}-\ref{Gamma curve}. Combining our Lemma \ref{generic holonomy} with Theorem 2.1 in \cite{hitchin_harmonic_1974} and \cite{wang_parallel_1989}, we obtain the following.
\begin{theorem}
\label{generic}
All negative Einstein metrics in Theorem \ref{zeta curve}-\ref{Gamma curve} does not have any parallel spinors. Ricci-flat metrics $\zeta_{(s_1,s_2,0)}$ and $\Gamma_0$ on $M^8$,
Ricci-flat metrics $\gamma_{\left(\frac{1}{\sqrt{5}},\frac{2}{\sqrt{5}},0\right)}$ on $\mathbb{R}^8$ have 1 parallel spinor.
All the other ALC Ricci-flat metrics in Theorem \ref{zeta curve}-\ref{Gamma curve} does not have any parallel spinor.
\end{theorem}

In particular for $m=1$, the continuous family of Ricci-flat metrics $\gamma_{(s_1,s_2,0)}$ has the $\spin(7)$ metric $\mathbb{A}_8$ lies in the interior and all the other Ricci-flat metrics have generic holonomy. Hence the parallel spinor on $\mathbb{A}_8$ is not preserved under a continuous deformation of Ricci-flat metrics through the family $\gamma_{(s_1,s_2,0)}$. Such a phenomenon also occurs for $G_2$ holonomy \cite{chi_invariant_2019}.
Parallel spinors are preserved under a continuous deformation of Ricci-flat metrics if the manifold is compact. Please see Theorem A in \cite{wang_preserving_1991} for more details.

Principal orbit of manifolds studied in this article are from the group triple $(G,H,K)$ given by
$$
(Sp(m+1)U(1),Sp(m)Sp(1)U(1),Sp(m)\Delta U(1)).
$$
The principal orbit is the total space of quaternionic Hopf fibration
\begin{equation}
\label{quaterionic hopf}
\mathbb{S}^3\hookrightarrow \mathbb{S}^{4m+3}\rightarrow \mathbb{HP}^m.
\end{equation}
Take $\mathbb{S}^{4m+3}$ as the space of unit quaternionic vector in $\mathbb{H}^{m+1}$. The fibration $\mathbb{S}^{4m+3}\rightarrow \mathbb{HP}^{m}$ is given by $(q_1,\dots,q_{m+1})\mapsto [q_1:\ldots:q_{m+1}]$. 
The transitive action of $G$ on $\mathbb{S}^{4m+3}$ is given by
\begin{equation}
\label{group action}
(A,z)\cdot \mathbf{q}
:=A\mathbf{q}\bar{z}
\end{equation}
for each $(A,z)\in G$. The isotropy group for $(0,\dots,0,1)\in\mathbb{S}^{4m+3}$ is $K$. The action of $G$ passes down to the base. The isotropy group for $[0:\ldots:0:1]$ is $H$. Therefore, the quaternionic Hopf fibration is indeed the homogeneous fibration
$
H/K\hookrightarrow G/K\rightarrow G/H.
$
More details of the isotropy representation are discussed in the next section.

Let $M$ be the cohomogeneity one manifold with principal orbit $G/K$ and singular orbit $G/H$. Then $M$ is an $\mathbb{R}^4$ bundle over $\mathbb{HP}^m$. A cohomogeneity one metric on $M$ has the form of $dt^2+g_{G/K}(t)$, where $g_{G/K}(t)$ is an invariant metric on each $\{t\}\times G/K$ with $t>0$ and it collapse to an invariant metric on $G/H$ as $t\to 0$. We also construct cohomogeneity one Einstein manifolds where the singular orbit for these manifolds is a singleton. In that scenario, the homogeneous part $g_{G/K}(t)$ vanishes as $t\to 0$. Since the principal orbit is $\mathbb{S}^{4m+3}$, the cohomogeneity one manifold is topologically $\mathbb{R}^{4m+4}$. 

One feature of the case in this article that differs from the one in \cite{chi_invariant_2019} is that the singular orbit is irreducible and the fiber is of two irreducible summands. Moreover, irreducible summands in $\mathfrak{g}/\mathfrak{k}$ all have different dimensions, as shown in Section \ref{Einstein Equation}. The cohomogeneity one dynamic systems have less symmetry than the one in \cite{chi_invariant_2019}. It is worth mentioning that the cohomogeneity one equation in the article shares some degree of similarity with the one that appears in \cite{reidegeld_exceptional_2011}. The study may help shed some light on the global existence question of $\spin(7)$ metric with an Aloff--Wallach space as the principal orbit.

\begin{remark}
\label{group L}
There exists an intermediate group $L:=Sp(m)U(1)U(1)$ between $H$ and $K$. With the same group action \eqref{group action} of $G$, we can see that the group triple $(G,L,K)$ gives the complex Hopf fibration
\begin{equation}
\label{complex hopf}
\mathbb{S}^1\hookrightarrow \mathbb{S}^{4m+3}\rightarrow \mathbb{CP}^{2m+1}.
\end{equation}
Let $\tilde{M}$ be the vector bundle with principal orbit $G/K$ and singular orbit $G/L$. It is a natural question to ask if there are more complete cohomogeneity one Einstein metrics on $\tilde{M}$ besides those constructed in \cite{berard-bergery_sur_1982}. Specifically, isotropy representation of $G/L$ has two irreducible summands that allow each $\{t\}\times G/L$ with $t>0$ to be squashed and $g_{G/K}(t)$ is a $G$-invariant metric on a circle bundle over a squashed $\mathbb{CP}^{2m+1}$.
\end{remark}

The Einstein metrics constructed and recovered in this article have three kinds of asymptotic behaviors. We give definitions in the following.
\begin{definition}
\label{Asymp conical}
Let $(M,g_M)$ be a Riemannian manifold of dimension $n+1$. Let $(N,g_N)$ be an $n$-dimensional Riemannian manifolds and $(C(N),dt^2+t^2g_N)$ be the metric cone with base $N$. Let $\bullet$ denote the tip of the cone. $M$ is \emph{asymptotically conical (AC)} if for some $p\in M$, we have $\lim\limits_{l\to \infty}((M,p),\frac{1}{l}g_M)=((C(N),\bullet),dt^2+t^2g_N)$ in the pointed Gromov--Hausdorff sense.
\end{definition}
\begin{remark}
Note that if $(N,g_N)$ in Definition \ref{Asymp conical} is a standard sphere $\mathbb{S}^n$, the metric $dt^2+t^2g_N$ is the Euclidean metric on $\mathbb{R}^{n+1}$. Then $M$ is \emph{asymptotically Euclidean (AE)}.
\end{remark}
\begin{definition}
Let $(M,g_M)$ be a Riemannian manifold of dimension $n+2$. Let $(N,g_N)$ be an $n$-dimensional Riemannian manifolds and $(C(N),dt^2+t^2g_N)$ be the metric cone with base $N$. $M$ is \emph{asymptotically locally conical (ALC)} if for some $p\in M$, we have $\lim\limits_{l\to \infty}((M,p),\frac{1}{l}g_M)=(\hat{C}(N),\bullet),dt^2+C ds^2+t^2g_N)$ in the pointed Gromov--Hausdorff sense, where $\hat{C}(N)$ is some $\mathbb{S}^1$-bundle over $C(N)$ and $C>0$ is a constant.
\end{definition}
\begin{definition}
\label{Asymp Hyperbolic}
Let $(M,g_M)$ be a Riemannian manifolds of dimension $n+1$ with a boundary $\partial M$. $M$ is \emph{conformally compact} if there exists a positive function $f$ such that $(M,f^2g_M)$ extends to a smooth metric on $\overline{M}$.
\end{definition}
In Definition \ref{Asymp Hyperbolic}, it can be checked that sectional curvature of $g_M$ approaches to $-\|df\|_{f^2g_M}$ near $\partial M$. If $(M,g_M)$ is negative Einstein, then the sectional curvature must approach to a constant near $\partial M$. With normalization, we fix $\|df\|_{f^2g_M}=1$. Hence a conformally compact Einstein manifold is also called an \emph{asymptotically hyperbolic} (AH) manifold.

This article is structured as the following. In Section \ref{Einstein Equation}, we derive the cohomogeneity one Einstein equation with principal orbit $G/K$. Then finding a cohomogeneity one Einstein metric is equivalent to finding an integral curve defined on $[0,\infty)$. 
Then we apply coordinate change inspired by the one in \cite{dancer_non-kahler_2009}\cite{dancer_new_2009}. In the new coordinate, initial conditions and the asymptotic limits of the original system are transformed to critical points. Then the construction of Einstein metrics boils down to finding integral curves that emanate from one critical point and tend to the other. Proving the completeness of the metric is equivalent to showing that the new integral curve is defined on $\mathbb{R}$.

In Section \ref{Critical Points}, we compute linearizations of some critical points with geometric significance of the new system. There are three critical points that represents different initial conditions. One of them gives the smooth extension of the metric to $G/H$; one gives the smooth extension of the metric to the origin of $\mathbb{R}^{4m+4}$; and third one gives the singular extension to the origin of $\mathbb{R}^{4m+4}$. There are two types of critical points that represent different asymptotic limits. One of them represents the ALC limit and the other type serves as the AH limit for the integral curves.

In Section \ref{Compact invariant set}, we construct a compact invariant set that contains sellected critical points in the previous section  on its boundary. Linearization in the previous section helps to prove that some integral curves that emanate from these points are in the compact invariant set initially. Hence the completeness  of the represented metrics follows. The technique we use is very similar to the one in \cite{chi_cohomogeneity_2019}.

In Section \ref{Asymptotic}, we give a rigorous proof for the asymptotic behaviour of the complete integral curves. We prove that all the new Ricci-flat metrics constructed are ALC, generalizing the conclusion in \cite{cvetic_new_2004} and \cite{bazaikin_new_2007}. We also prove that all the new negative Einstein metrics constructed are AH.

In Section \ref{Holonomy}, we recover some well-known examples, including the K\"ahler--Einstein metrics in \cite{berard-bergery_sur_1982}, the quaternionic K\"ahler metric in \cite{swann_hyper-kahler_1991},
the $G_2$ and $\spin(7)$ metrics in \cite{bryant_construction_1989}\cite{gibbons_einstein_1990} and the $\spin(7)$ metric in \cite{cvetic_new_2004}. Then we show that all the ALC Ricci-flat metrics constructed have generic holonomy.

\section{Einstein Equation}
\label{Einstein Equation}
We derive the Einstein equation in this section. We first introduce some notation regarding representation theory. Let $\mathbb{I}$ be the trivial representation. Let $\mu_m$ be the matrix multiplication representation (over complex numbers) of $Sp(m)$. Let $\sigma^l$ denote the $l$-th symmetric tensor power of the matrix multiplication representation (over complex numbers) of $Sp(1)$ (hence $\mu_1=\sigma^1$). Let $\mathfrak{t}^l$ denote the complex representation of $U(1)$ of weight $l$. Define the inner product for $\mathfrak{g}$ as $Q(X,Y)=-\trace(XY)$. Note that $Q$ is non-degenerate on $\mathfrak{g}$ and equal to a multiple of the Killing form of $\mathfrak{sp}(m+1)$ when restricted to $Sp(m+1)$.

The action of $G$ on $T_{eK}(G/K)$ is equivalent to the adjoint action of $G$ on $\mathfrak{g}/\mathfrak{k}$. Let $(G,H,K)=(Sp(m+1)U(1),Sp(m)Sp(1)U(1),Sp(m)\Delta U(1))$. We have the following $Q$-orthogonal decomposition for $\mathfrak{g}$.
\begin{equation}
\begin{split}
\mathfrak{g}&=\mathfrak{h}\oplus [\mu_m\otimes \sigma^1]_\mathbb{R}\quad \text{as a $H$-module}\\
&=\mathfrak{l}\oplus [\mathfrak{t}^2]_\mathbb{R}\oplus  [\mu_m\otimes \mathfrak{t}^{1}]_\mathbb{R}\quad \text{as an $L$-module}\\
&=\mathfrak{k}\oplus \mathbb{I}\oplus [\mathfrak{t}^2]_\mathbb{R}\oplus  [\mu_m\otimes \mathfrak{t}^{1}]_\mathbb{R}\quad \text{as a $K$-module}
\end{split}.
\end{equation}
Consider $Sp(m+1)=U(2m+2)\cap Sp(2m+2; \mathbb{C})$ and embed $G$ in $Sp(2m+4;\mathbb{C})$. Identify $\mathbb{H}^{m+1}$ with $\mathbb{C}^{2m+2}=\mathbb{C}^{m+1}\oplus j\mathbb{C}^{m+1}$.
The isotropy representation of $G/K$ hence has a $Q$-orthonormal basis $\{E_1,E_2,E_3,E_{l_1 l_2}\mid l_1=1,\dots,m,\quad l_2=1,2,3,4\}$, where \begin{equation}
\begin{split}
&E_1=\frac{1}{2}\begin{bmatrix}
O&&&&&&\\
&\ddots&&&&&\\
&&O&&&&\\
&&&-i&0&&\\
&&&0&i&&\\
&&&&&i&0\\
&&&&&0&-i\\
\end{bmatrix},\\
&E_2=\frac{1}{\sqrt{2}}\begin{bmatrix}
O&&&\\
&\ddots&&&&\\
&&O&&&\\
&&&0&1&\\
&&&-1&0&\\
&&&&&O\\
\end{bmatrix},\quad 
E_3=\frac{1}{\sqrt{2}}\begin{bmatrix}
O&&&&&\\
&\ddots&&&&\\
&&O&&&\\
&&&0&i&\\
&&&i&0&\\
&&&&&O\\
\end{bmatrix},
\end{split}
\end{equation} and each $E_{l_1 l_2}$ is given by
$$
E_{l_1 l_2}=\frac{1}{2}\begin{bmatrix}
O&&&&O\\
&\ddots&&&\vdots\\
&&\ddots&&A_{l_2}\\
&&&\ddots&\vdots\\
O&\dots&-A_{l_2}^*&\dots&O\\
&&&&&O\\
\end{bmatrix},\quad l_1=1,\dots, m,\quad  l_2=1,2,3,4,
$$
with
$$O=\begin{bmatrix}
0&\\
&0
\end{bmatrix},\quad
A_1=\begin{bmatrix}
1&0\\
0&1
\end{bmatrix},\quad
A_2=\begin{bmatrix}
i&0\\
0&-i
\end{bmatrix},\quad A_3=\begin{bmatrix}
0&i\\
i&0
\end{bmatrix},\quad A_4=\begin{bmatrix}
0&-1\\
1&0
\end{bmatrix}.
$$
The trivial representation $\mathbb{I}$ is spanned by $E_1$, which is orthogonal to $\mathfrak{k}$. Note that $Q=-\frac{1}{4}B_1=-\frac{1}{2m+4}B_2$, where $B_1$ and $B_2$ are respectively the Killing form for $Sp(2;\mathbb{C})$ and $Sp(2m+2;\mathbb{C})$. We abuse the notation by using $Q$ to denote the invariant metric on $G$ that is induced by the inner product. Take $Q$ as the background metric.  By Schur's Lemma, an invariant metric on $G/K$ has the form of 
\begin{equation}
\label{invariant metric}
g_{G/K}=a^2\left.Q\right|_{\mathbb{I}}+b^2\left.Q\right|_{[\mathfrak{t}^2]_\mathbb{R}}+c^2\left.Q\right|_{\mathfrak{g}/\mathfrak{h}}.
\end{equation}
By Corollary 7.39 in \cite{besse_einstein_2008}, the formula of the scalar curvature for $g_{G/K}$ is 
$$
R_s=\frac{4}{b^2}+\frac{4m(m+2)}{c^2}-\frac{1}{2}\frac{a^2}{b^4}-\frac{m}{4}\dfrac{a^2}{c^4}-m\dfrac{b^2}{c^4}.
$$
Compute the first variation of the Hilbert--Einstein functional on $G/K$. The Ricci endomorphism is given by 
\begin{equation}
\label{ricci endo}
\begin{split}
r_a&=\frac{1}{2}\frac{a^2}{b^4}+\frac{m}{4}\frac{a^2}{c^4}\\
r_b&=\frac{2}{b^2}-\frac{1}{2}\frac{a^2}{b^4}+\frac{m}{2}\frac{b^2}{c^4}\\
r_c&=\frac{m+2}{c^2}-\frac{1}{8}\frac{a^2}{c^4}-\frac{1}{2}\frac{b^2}{c^4}
\end{split}.
\end{equation}

Note that $M\backslash (G/H)$ and $\mathbb{R}^{4m+4}\backslash\{0\}$ are both $G$-diffeomorphic to $(0,\infty)\times G/K$. We construct Einstein metrics $g=dt^2+g_{G/K}(t)$ by setting $(0,\infty)$ as a geodesic and assign $G$-invariant metric $g_{G/K}$ on each $\{t\}\times G/K$. Then \eqref{invariant metric} becomes a $S^2(\mathfrak{g}/\mathfrak{k})^K$-valued function on $t$, where $S^2(\mathfrak{g}/\mathfrak{k})^K$ is the space of $K$-invariant symmetric 2-tensor. By \cite{eschenburg_initial_2000}, the cohomogeneity one Einstein system is
\begin{equation}
\label{Old Einstein Equation}
\begin{split}
\frac{\ddot{a}}{a}-\left(\frac{\dot{a}}{a}\right)^2&=-\left(\frac{\dot{a}}{a}+2\frac{\dot{b}}{b}+4m\frac{\dot{c}}{c}\right)\frac{\dot{a}}{a}+\frac{1}{2}\frac{a^2}{b^4}+\frac{m}{4}\frac{a^2}{c^4}-\Lambda\\
\frac{\ddot{b}}{b}-\left(\frac{\dot{b}}{b}\right)^2&=-\left(\frac{\dot{a}}{a}+2\frac{\dot{b}}{b}+4m\frac{\dot{c}}{c}\right)\frac{\dot{b}}{b}+\frac{2}{b^2}-\frac{1}{2}\frac{a^2}{b^4}+\frac{m}{2}\frac{b^2}{c^4}-\Lambda\\
\frac{\ddot{c}}{c}-\left(\frac{\dot{c}}{c}\right)^2&=-\left(\frac{\dot{a}}{a}+2\frac{\dot{b}}{b}+4m\frac{\dot{c}}{c}\right)\frac{\dot{c}}{c}+\frac{m+2}{c^2}-\frac{1}{8}\frac{a^2}{c^4}-\frac{1}{2}\frac{b^2}{c^4}-\Lambda\\
\end{split}
\end{equation}
with conservation law
\begin{equation}
\label{Old Conservation Law}
\begin{split}
&\left(\frac{\dot{a}}{a}+2\frac{\dot{b}}{b}+4m\frac{\dot{c}}{c}\right)^2-\left(\frac{\dot{a}}{a}\right)^2-2\left(\frac{\dot{b}}{b}\right)^2-4m\left(\frac{\dot{c}}{c}\right)^2=R_s-(4m+2)\Lambda
\end{split}.
\end{equation}

There are three possible initial conditions for \eqref{Old Einstein Equation}. The first possibility is having $G/H$ as the singular orbit. The cohomogeneity one manifold $M$ is an $\mathbb{R}^4$-bundle over $\mathbb{HP}^m$. The principal orbit $G/K$ becomes the zero section $G/H$ as $t\to 0$. In order to smoothly extend the metric on the tubular neighbourhood around $G/H$, we have the following proposition.

\begin{proposition}
\label{3-sphere bundle initial condition}
The necessary and sufficient conditions for the metric $g=dt^2+g_{G/K}(t)$ to extend smoothly to a metric in a tubular neighborhood of $G/H$ is
\begin{equation}
\label{initial condition G/H}
\lim_{t\to 0} (a,b,c,\dot{a},\dot{b},\dot{c})=\left(0,0,h,1,\frac{\sqrt{2}}{2},0\right)
\end{equation}
for some $h>0$.
\end{proposition}
\begin{proof}
Since the unit sphere in $\mathfrak{q}_+$ is generated by $E_1$, $E_2$ and $E_3$.
It is clear that $\left.Q\right|_{\mathbb{I}}+\frac{1}{2}\left.Q\right|_{[\mathfrak{t}^2]_\mathbb{R}}$ is the standard metric for $H/K=\mathbb{S}^3$. The initial condition is then derived by Lemma 9.114 in \cite{besse_einstein_2008}.
\end{proof}

Another possible initial condition is $G/K$ collapsing to a singleton as $t\to 0$. Since $G/K=\mathbb{S}^{4m+3}$, the cohomogeneity one manifold is topologically $\mathbb{R}^{4m+4}$.  In order to extend the metric on the neighborhood of the origin of $\mathbb{R}^{4m+4}$, we have the following proposition.
\begin{proposition}
The necessary and sufficient conditions for the metric $g=dt^2+g_{G/K}(t)$ to extend smoothly to a metric in a tubular neighborhood of origin in $\mathbb{R}^{4m+4}$ is
\begin{equation}
\label{initial condition Taub-NUT-1}
\lim_{t\to 0} (a,b,c,\dot{a},\dot{b},\dot{c})=\left(0,0,0,1,\frac{\sqrt{2}}{2},\frac{1}{2}\right).
\end{equation}
\end{proposition}
\begin{proof}
The unit sphere $\mathbb{S}^{4m+3}$ is generated by $E_1,E_2,E_3$ and $E_{lj}$'s.
Therefore, if $$g_{G/K}(t)=t^2\left(\left.Q\right|_{\mathbb{I}}+\frac{1}{2}\left.Q\right|_{[\mathfrak{t}^2]_\mathbb{R}}+\frac{1}{4}\left.Q\right|_{\mathfrak{g}/\mathfrak{h}}\right),$$
$g=dt^2+g_{G/K}(t)$ is the flat metric on $\mathbb{R}^{4m+4}$. The initial condition is obtained by Lemma 9.114 in \cite{besse_einstein_2008}.
\end{proof}

Note that $G/K$ admits two homogeneous Einstein metrics. Hence for a cohomogeneity one metric of Taub-NUT type, $G/K$ can also degenerate to a point as a Jensen sphere \cite{jensen_einstein_1973}. Then the corresponding initial condition is given by
\begin{equation}
\label{initial condition Taub-NUT-2}
\lim_{t\to 0} (a,b,c,\dot{a},\dot{b},\dot{c})=\left(0,0,0,\beta,\frac{\sqrt{2}}{2}\beta,\frac{\sqrt{2m+3}}{2}\beta\right),
\end{equation}
where $(4m+3)(4m+2)\beta^2=6+\frac{16m(m+2)(2m+3)-12m}{(2m+3)^2}$. In other words,  if 
$$g_{G/K}(t)=\beta^2t^2\left(\left.Q\right|_{\mathbb{I}}+\frac{1}{2}\left.Q\right|_{[\mathfrak{t}^2]_\mathbb{R}}+\frac{2m+3}{4}\left.Q\right|_{\mathfrak{g}/\mathfrak{h}}\right),$$ then $dt^2+g_{G/K}(t)$ is a singular cone metric on $\mathbb{R}^{4m+4}$ with the Jensen sphere $\mathbb{S}^{4m+3}$ as its base.

As pointed out in Remark 2.9 in \cite{chi_invariant_2019}, in the Ricci-flat case, changing $h$ in \eqref{initial condition G/H} is essentially the homothetic change of the solution around $G/H$. Moreover, \eqref{initial condition G/H} \emph{does not} fully determine the metric in a tubular neighborhood of $G/H$. This is also the case for \eqref{initial condition Taub-NUT-1}. Using Lemma 1.1 in \cite{eschenburg_initial_2000}, we can prove that there exists a free parameter for $a-b$ of order 3 for \eqref{initial condition G/H} and \eqref{initial condition Taub-NUT-1}. We consider \eqref{initial condition G/H} bellow. Statements concerning \eqref{initial condition Taub-NUT-1} can be obtained without substantial change of the argument.

We first rephrase Lemma 1.1 in \cite{eschenburg_initial_2000} for $M$ below.
\begin{lemma}[\cite{eschenburg_initial_2000}]
\label{EW}
Let $\chi$ be the slice representation for $M$. Let $W_i=\hmo(S^i(\chi),S^2(\chi\oplus \mathfrak{g}/\mathfrak{h}))^H$ be the space of $H$-equivariant homogeneous polynomials of degree $i$. 
Consider a smooth curve $g(t): [0,\infty)\to S^2(\chi\oplus \mathfrak{g}/\mathfrak{h})^K$ with Taylor expansion $\sum_{i=0}^\infty g_it^i$ around $t=0$. The curve can be smoothly extended to $G/H$ as a symmetric 2-tensor if and only if each $g_i$ is an evaluation of some element in $W_i$ at $v_0=(1,0,0,0)\in \chi$.
\end{lemma}

Since $\chi=\left[\sigma^1\otimes \mathfrak{t}^{1}\right]_\mathbb{R}$ and $\mathfrak{g}/\mathfrak{h}=\left[\mu_m\otimes \sigma^1\right]_\mathbb{R}$ are inequivalent, we have decomposition
$$
W_i=W_i^+\oplus W_i^-:=\hmo(S^i(\chi),S^2(\chi))^H\oplus \hmo(S^i(\chi),S^2(\mathfrak{g}/\mathfrak{h}))^H
$$
By induction, we have
\begin{equation}
\begin{split}
S^{2k}(\chi)\otimes \mathbb{C}&=\sum_{l=0}^{k-1} \sum_{j=0}^{l} \left(\sigma^{2k-2j}\otimes \mathfrak{t}^{2k-2l}+ \sigma^{2k-2j}\otimes \mathfrak{t}^{-(2k-2l)}\right)+\sum_{l=0}^{k}\sigma^{2k-2l}\\
S^{2k+1}(\chi)\otimes \mathbb{C}&=\sum_{l=0}^{k} \sum_{j=0}^{l} \left(\sigma^{2k+1-2j}\otimes \mathfrak{t}^{2k+1-2l}+ \sigma^{2k+1-2j}\otimes \mathfrak{t}^{-(2k+1-2l)}\right)
\end{split}
\end{equation}
as $H$-modules. In particular, we have
$$
S^2(\chi)=[\sigma^2\otimes \mathfrak{t}^2]_\mathbb{R}+[\sigma^2]_\mathbb{R}+\mathbb{I}.
$$
We also have
\begin{equation}
S^2(\mathfrak{g}/\mathfrak{h})=\left\{\begin{array}{ll}
\left[\mu_m^2\otimes \sigma^2\right]_\mathbb{R}+\left[\mu_m\mathring{\wedge}\mu_m\right]_\mathbb{R}+\mathbb{I}& m\neq 1\\
\left[\sigma^2\otimes \sigma^2\right]_\mathbb{R}+\mathbb{I}& m=1\\
\end{array}\right.,
\end{equation}
where $\left[\mu_m \mathring{\wedge} \mu_m\right]_\mathbb{R}+\mathbb{I}=\left[\mu_m\wedge\mu_m\right]_\mathbb{R}$. 
Hence it is clear that
$$
W_{2k}^+=\left\{\begin{array}{ll}
\mathbb{R}& k=0\\
\mathbb{R}^3& k\geq 1
\end{array}\right. \quad 
W_{2k+1}^+=0 \quad
W_{2k}^-=\mathbb{R}\quad W_{2k+1}^-=0.
$$
\begin{proposition}
\label{more parameter!}
For initial condition \eqref{initial condition G/H}, there exists a free parameter for $a-b$ of order 3.
\end{proposition}
\begin{proof}
Identify $g=dt^2+g_{G/K}(t)$ as a map $D(t)\oplus J(t)$, where $D(t)\colon [0,\infty)\to S^2(\chi)^K$ and $J(t)\colon [0,\infty)\to S^2(\mathfrak{g}/\mathfrak{k})^K$. In that way, the standard inner product on each fiber $\chi$ is given by $dt^2+t^2(\left.Q\right|_{\mathbb{I}}+\frac{1}{2}\left.Q\right|_{[\mathfrak{t}^2]_\mathbb{R}})$.

The Taylor expansion can be written as
\begin{equation}
\begin{split}
D(t)&=D_0+D_1t+D_2t^2+\dots\\
J(t)&=J_0+J_1t+J_2t^2+\dots
\end{split}.
\end{equation}
Since $W^-_{2k+2}\cong W^-_{2k}$ for $k\geq 0$ and $W^-_0$ is spanned by the identity matrix, we learn that $J(t)$ is determined by $J_0=h^2\id$. Hence no free variable of higher order comes from the $c$ component.

The generator for $W^+_0$ is the identity matrix $\id$.  Hence one of the generators of $W^+_2$ is $(\sum_{i=0}^3 x_i^2)\id$. Note that the identity map in $W^+_{2}$ is clearly $H$-equivariant. Hence the matrix $\Pi$, where $\Pi_{ij}=x_ix_j$ is another generator of $W^+_2$. By straightforward computation, the third generator of $W^+_2$ is $\Xi$ the projection map from $S^{2}(\chi)$ to the 3-dimensional subspace of $S^2(\chi)$.
$$
\Xi=
\begin{bmatrix}
x_1^2+x_2^2-x_3^2-x_4^2&0&2(x_2x_4-x_1x_3)&-2(x_1x_4+x_2x_3)\\
0&x_1^2+x_2^2-x_3^2-x_4^2
&2(x_1x_4+x_2x_3)&2(x_2x_4-x_1x_3)\\
2(x_2x_4-x_1x_3)&2(x_1x_4+x_2x_3)&-x_1^2-x_2^2+x_3^2+x_4^2
&0\\
-2(x_1x_4+x_2x_3)&2(x_2x_4-x_1x_3)&0&-x_1^2-x_2^2+x_3^2+x_4^2.
\end{bmatrix}.
$$ 
Evaluate these three generators at $v_0$ and take into account that $t$ is a unit speed geodesic. We learn that $D_0=\id$ and 
$D_2$ is a multiple of 
$$
p\left(\left(\sum_{i=0}^3 x_i^2\right)\id-\Pi\right)(v_0)+q\left(\left(\sum_{i=0}^3 x_i^2\right)\id-\Xi\right)(v_0)=
\begin{bmatrix}
0&&&\\
&p&&\\
&&p+2q&\\
&&&p+2q
\end{bmatrix}
$$
for some $p,q\in \mathbb{R}$. Since $W_2^+/W_0^+\cong \mathbb{R}^2$, there are in principle two free variables for $D(t)$ to extend smoothly around $G/H$ as a 2-tensor. However, with the geometric setting that $t$ is a unit geodesic, the parameter $p$ is determined. Therefore, $g$ can be extended smoothly around $G/H$ if
\begin{equation}
\begin{split}
a^2&=t^2+A t^4+O(t^6)\\
b^2&=t^2+B t^4+O(t^6)\\
c^2&=h^2+O(t^2),
\end{split}
\end{equation}
where $(\dddot{a}-\dddot{b})(0)=3(A-B)=-3q$ for some $q\in \mathbb{R}$.
\end{proof}

\begin{remark}
Proposition \ref{more parameter!} can be carried over to \eqref{initial condition Taub-NUT-1} by thinking $\mathbb{R}^{4m+4}$ as a vector bundle over a singleton. In this case, $K$ is the isotropy representation at $(1,0,\dots,0)$. The space to consider is $\hmo(S^i(\tilde{\chi}),S^2(\tilde{\chi}))^G$, where $\tilde{\chi}$ is the slice representation by the action of $G$. Lemma \ref{EW} can then be applied with no extra difficulties. Besides the discussion above, there is an alternative procedure to derive the smoothness condition. More details are presented in \cite{verdiani_smoothness_2020}.
\end{remark}

Inspired by \cite{dancer_non-kahler_2009}\cite{dancer_new_2009}, we apply coordinate change $d\eta=\left(\frac{\dot{a}}{a}+2\frac{\dot{b}}{b}+4m\frac{\dot{c}}{c}\right)dt$. The quantity $\frac{\dot{a}}{a}+2\frac{\dot{b}}{b}+4m\frac{\dot{c}}{c}$ is the trace of the shape operator of the hypersurface orbit. Define
\begin{equation}
\label{new variable}
\begin{split}
&X_1=\frac{\frac{\dot{a}}{a}}{\frac{\dot{a}}{a}+2\frac{\dot{b}}{b}+4m\frac{\dot{c}}{c}},\quad X_2=\frac{\frac{\dot{b}}{b}}{\frac{\dot{a}}{a}+2\frac{\dot{b}}{b}+4m\frac{\dot{c}}{c}},\quad X_3=\frac{\frac{\dot{c}}{c}}{\frac{\dot{a}}{a}+2\frac{\dot{b}}{b}+4m\frac{\dot{c}}{c}},\\
&Y_1=\frac{a}{b},\quad Y_2=\frac{\frac{1}{b}}{\frac{\dot{a}}{a}+2\frac{\dot{b}}{b}+4m\frac{\dot{c}}{c}},\quad Y_3=\frac{\frac{b}{c^2}}{\frac{\dot{a}}{a}+2\frac{\dot{b}}{b}+4m\frac{\dot{c}}{c}},\quad \tilde{W}=\frac{1}{\frac{\dot{a}}{a}+2\frac{\dot{b}}{b}+4m\frac{\dot{c}}{c}}.
\end{split}
\end{equation}
Define functions on $\eta$
\begin{equation}
\begin{split}
&R_1=\frac{1}{2} Y_1^2Y_2^2+\frac{m}{4} Y_1^2Y_3^2\\
&R_2=2Y_2^2-\frac{1}{2} Y_1^2Y_2^2+\frac{m}{2} Y_3^2\\
&R_3=(m+2)Y_2Y_3-\frac{1}{8}Y_1^2Y_3^2-\frac{1}{2}Y_3^2\\
&R_s=R_1+2R_2+4mR_3,\quad G=X_1^2+2X_2^2+4mX_3^2
\end{split}.
\end{equation}
Let $'$ denote the derivative with respect to $\eta$. The Einstein equations \eqref{Old Einstein Equation} become a polynomial system
\begin{equation}
\label{Polynomial Einstein Equation pre}
\begin{split}
\begin{bmatrix}
X_1\\
X_2\\
X_3\\
Y_1\\
Y_2\\
Y_3\\
 \tilde{W}
\end{bmatrix}'&=V(X_1,X_2,X_3,Y_1,Y_2,Y_3,\tilde{W})=\begin{bmatrix}
X_1(G+\Lambda  \tilde{W}^2-1)+R_1-\Lambda  \tilde{W}^2\\
X_2(G+\Lambda  \tilde{W}^2-1)+R_2-\Lambda  \tilde{W}^2\\
X_3(G+\Lambda  \tilde{W}^2-1)+R_3-\Lambda  \tilde{W}^2\\
Y_1(X_1-X_2)\\
Y_2(G+\Lambda  \tilde{W}^2-X_2)\\
Y_3(G+\Lambda  \tilde{W}^2+X_2-2X_3)\\
 \tilde{W}(G+\Lambda  \tilde{W}^2)
\end{bmatrix}\\
\end{split},
\end{equation}
with conservation law \eqref{Old Conservation Law} becomes
\begin{equation}
\label{Polynomial Conservation Law}
\begin{split}
\mathcal{C}:1-G&=R_s-(4m+2)\Lambda  \tilde{W}^2
\end{split}
\end{equation}
It is clear that $X_1+2X_2+4mX_3\equiv 1$ from the definition of coordinate change. In fact, let $$\mathcal{H}=\{(X_1,X_2,X_3,Y_1,Y_2,Y_3,\tilde{W})\mid X_1+2X_2+4m X_3=1\},$$
one can check that $\mathcal{C}\cap\mathcal{H}\cap \{ \tilde{W}\geq 0\}$ is a flow-invariant $5$-dimensional manifold in $\mathbb{R}^7$ with a $4$-dimensional boundary $\mathcal{C}\cap\mathcal{H}\cap \{ \tilde{W}\equiv 0\}$. 

\begin{remark}
\label{how to recover the origin}
For \eqref{Polynomial Einstein Equation pre} with $\Lambda<0$, the variable $t$ and functions $a$, $b$ and $c$ are recovered by
\begin{equation}
\label{recovery}
t=\int_{\eta_0}^\eta \tilde{W} d\tilde{\eta}, \quad a=\frac{Y_1 \tilde{W}}{Y_2},\quad b=\frac{ \tilde{W}}{Y_2},\quad c=\frac{ \tilde{W}}{\sqrt{Y_2Y_3}}.
\end{equation}
\end{remark}

\begin{remark}
\label{comment on W}
If we assume $\Lambda =0$ in \eqref{Polynomial Einstein Equation}. Since $ \tilde{W}'=G \tilde{W}$ in this case, we have 
$$
 \tilde{W}= \exp\left(\int_{\tilde{\eta}_0}^{\eta} G d\tilde{\eta}\right).
$$
Since $d\eta=\frac{1}{ \tilde{W}} dt=\exp\left(-\int_{\tilde{\eta}_0}^{\eta} G d\tilde{\eta}\right)dt$, the variable $t$ and functions $a$, $b$ and $c$ can be recovered without $ \tilde{W}$.  Therefore, for cohomogeneity one Ricci-flat metrics, we consider the vector field $V_{RF}$ on the 4-dimensional invariant manifold
$$\mathcal{C}_{RF}=\{(X_1,X_2,X_3,Y_1,Y_2,Y_3)\mid 1-G= R_1+2R_2+4mR_3,\quad  X_1+2X_2+4mX_3=1\}$$ 
given by $\eqref{Polynomial Einstein Equation pre}$ with all $\tilde{W}$ terms deleted.

On the other hand, it is clear \eqref{Polynomial Einstein Equation pre} has a subsystem restricted on $\mathcal{C}\cap\mathcal{H}\cap \{ \tilde{W}\equiv 0\}$.
Consider the map $\Psi: \mathcal{C}_{RF}\to \mathcal{C}$ by $(X_1,X_2,X_3,Y_1,Y_2,Y_3)\mapsto (X_1,X_2,X_3,Y_1,Y_2,Y_3,0)$. It is clear that $(\mathcal{C}_{RF},V_{RF})$ and $(\mathcal{C}\cap\mathcal{H},\left.V\right|_{\mathcal{C}\cap\mathcal{H}\cap \{ \tilde{W}\equiv 0\}} )$ are $\Psi$-related. Therefore, cohomogeneity one Ricci-flat metrics can be represented by integral curves on $\mathcal{C}\cap\mathcal{H}\cap \{ \tilde{W}\equiv 0\}$, even though the quantity $\frac{1}{\frac{\dot{a}}{a}+2\frac{\dot{b}}{b}+4m\frac{\dot{c}}{c}}$ does not actually vanish on the Ricci-flat manifold.
\end{remark}

\begin{remark}
\label{homothety}
Note that \eqref{Old Einstein Equation} is not invariant under homothety change if $\Lambda<0$. We fix $\Lambda=-(4m+3)$ in this article to fix the homothety for negative Einstein metrics.

If $\Lambda=0$ in \eqref{Old Einstein Equation}, then the original system is invariant under homothety change. The homothety change is transformed to the shifting of $\eta$ for an integral curve, while the graph of the integral curve remains unchanged. Combining with Remark \ref{comment on W}, we know that each integral curve for $V$ restricted on $\mathcal{C}\cap \mathcal{H}\cap \{\tilde{W}\equiv 0\}$ represents a solution in the original coordinate up to homothety.
\end{remark}

For a technical reason that is further discussed in Remark \ref{different eigenvalues} in Section \ref{Critical Points}, instead of studying system \eqref{Polynomial Einstein Equation pre} on $\mathcal{C}\cap\mathcal{H}$, we study a dynamic system that is equivalent to \eqref{Polynomial Einstein Equation pre}. Remark \ref{how to recover the origin}, Remark \ref{comment on W} and Remark \ref{homothety} are carried over.

On $\mathbb{R}^6$, define 
$$
\mathcal{E}=\{(X_1,X_2,X_3,Y_1,Y_2,Y_3)\mid 1-G-R_s\geq 0,\quad X_1+2X_2+4mX_3=1\}.
$$
It is a $5$-dimensional surface in $\mathbb{R}^6$ with a boundary. Define
\begin{equation}
\Phi \colon \mathcal{E}\to \mathcal{C}\cap\mathcal{H}\cap\{ \tilde{W}\geq 0\}.
\end{equation}
by sending $(X_1,X_2,X_3,Y_1,Y_2,Y_3)$ to $\left(X_1,X_2,X_3,Y_1,Y_2,Y_3, \sqrt{\frac{1-G-R_s}{-(4m+2)\Lambda}}\right)$. It is straightforward to check that $\Phi$ is a diffeomorphism. On $\mathcal{E}$, define function $W=\sqrt{\frac{1-G-R_s}{-(4m+2)\Lambda}}$. Consider the dynamic system
\begin{equation}
\label{Polynomial Einstein Equation}
\begin{bmatrix}
X_1\\
X_2\\
X_3\\
Y_1\\
Y_2\\
Y_3
\end{bmatrix}'=
V_{\Lambda\leq 0}(X_1,X_2,X_3,Y_1,Y_2,Y_3):=\begin{bmatrix}
X_1(G+\Lambda W^2-1)+R_1-\Lambda  W^2\\
X_2(G+\Lambda  W^2-1)+R_2-\Lambda W^2\\
X_3(G+\Lambda W^2-1)+R_3-\Lambda  W^2\\
Y_1(X_1-X_2)\\
Y_2(G+\Lambda W^2-X_2)\\
Y_3(G+\Lambda  W^2+X_2-2X_3)
\end{bmatrix}
\end{equation}
on $\mathcal{E}$. By straightforward computation, we have 
\begin{equation}
\label{boundary flow invariant}
(G+R_s)'=2(G+R_s-1)(G+\Lambda W^2),
\end{equation}
from which we deduce $$W'=W(G+\Lambda W^2).$$   
Therefore, the boundary 
$$
\partial \mathcal{E}:=\{(X_1,X_2,X_3,Y_1,Y_2,Y_3)\mid 1-G-R_s=0,\quad  X_1+2X_2+4mX_3=1\}
$$
is flow-invariant. Moreover, $(\mathcal{E}, V_{\Lambda\leq 0})$ and $(\mathcal{C}\cap\mathcal{H}\cap \{\tilde{W}\geq 0\},V)$ are $\Phi$-related. We have the following commutative diagram.
\begin{equation}
\label{All related}
\begin{CD}
 \left(\mathcal{C}\cap\mathcal{H}\cap\{\tilde{W}\equiv 0\},\left. V\right|_{\mathcal{C}\cap\mathcal{H}\cap\{\tilde{W}\equiv 0\}}\right) @>>> (\mathcal{C}\cap\mathcal{H}\cap\{\tilde{W}\geq 0\},V)\\
@AA\left.\Phi\right|_{\partial \mathcal{E}} A @   AA\Phi A\\
( \partial \mathcal{E},\left. V_{\Lambda\leq 0}\right|_{\partial \mathcal{E}})    @>>>   ( \mathcal{E},V_{\Lambda\leq 0})
\end{CD}
\end{equation}

The variable $t$ and functions $a$, $b$ and $c$ can be recovered by replacing $\tilde{W}$ with $W$ in Remark \ref{how to recover the origin} and Remark \ref{comment on W}. By Remark \ref{comment on W} and Remark \ref{homothety}, we fix $\Lambda=-(4m+3)$ in $V_{\Lambda \leq 0}$ in order to fix the homothety for negative Einstein metrics. Each integral curve for $V_{\Lambda \leq 0}$ restricted on $\partial \mathcal{E}$ represents a Ricci-flat solution in the original coordinate up to homothety. Define $\mathcal{P}=\{(X_1,X_2,X_3,Y_1,Y_2,Y_3)\mid Y_1,Y_2,Y_3\geq 0\}$. It is clear that $\mathcal{E}\cap \mathcal{P}$ is flow-invariant.  By the discussion above, it is justified to denote $\partial \mathcal{E}\cap \mathcal{P}$ as $\mathcal{B}_{RF}$.

\begin{proposition}
\label{big eta means big t}
If $\Lambda=0$ in \eqref{Old Einstein Equation}, the solution for the original system is defined on $(0,\infty)$ if the corresponding integral curve is defined on $\mathbb{R}$. If $\Lambda<0$ in \eqref{Old Einstein Equation}, the solution for the original system is defined on $(0,\infty)$ if the corresponding integral curve is defined on $\mathbb{R}$ and $R_s\geq 0$ along the curve.
\end{proposition}
\begin{proof}
The Ricci-flat case was proven in Lemma 5.1 \cite{buzano_family_2015}. As for the negative Einstein case, since $R_s\geq 0$ along the corresponding integral curve, it is clear that $W$ is increasing along the curve. Hence we have $\lim\limits_{\eta\to\infty}t=\infty$. The proof is complete.
\end{proof}

To some extent, by the proposition above, the problem of constructing a cohomogeneity one Einstein metric $dt^2+g_{G/K}(t)$ on $(0,\infty)\times G/K$ is transformed to finding an integral curve of \eqref{Polynomial Einstein Equation} on $\mathcal{E}$ that is defined on $\mathbb{R}$. The initial conditions at $t=0$ are transformed to limits of these integral curves as $\eta\to -\infty$. In Section \ref{Critical Points}, we see that initial conditions \eqref{initial condition G/H}, \eqref{initial condition Taub-NUT-1} and \eqref{initial condition Taub-NUT-2} are transformed to critical points of the new system. Hence the next step is to show that integral curves that emanate from theses critical points are defined on $\mathbb{R}$.

There are some integral curves already known to be defined on $\mathbb{R}$. These curves lie in several subsystems of \eqref{Polynomial Einstein Equation} besides $\mathcal{B}_{RF}$. We give a short summary in the following.

Straightforward computation shows that $$\mathcal{B}_{Rd}:=\mathcal{E}\cap\mathcal{P}\cap \{X_1-X_2\equiv 0, Y_1^2 \equiv 2\}$$ is flow-invariant. Integral curves on this set represents metrics with $a^2\equiv 2b^2$ imposed. Hence the 3-sphere $H/K$ is round (hence the subscript ``Rd'') and the subsystem is of two summands type. This case is studied in \cite{wink_cohomogeneity_2017}\cite{bohm_non-compact_1999}. Furthermore, for $m=1$, there exists an integral curve that represents the $\spin(7)$ metric in \cite{bryant_construction_1989}\cite{gibbons_einstein_1990}. The metric can be represented by a straight line in terms of variables in \eqref{new variable}.

One can also see that $$\mathcal{B}_{FS}:=\mathcal{E}\cap\mathcal{P}\cap \{2Y_2-Y_3\equiv 0, X_2-X_3\equiv 0\}$$ is flow-invariant. Integral curves on this set represents cohomogeneity one metrics with $b^2\equiv 2c^2$ imposed. Under this setting, the homogeneous metric on $\mathbb{CP}^{2m+1}$ is the Fubini--Study metric and it is K\"ahler--Einstein. The imposed equation is also part of the K\"ahler condition shown in \cite{dancer_kahler-einstein_1998}. The circle bundle $\mathrm{Prin}(k)$ over $\mathbb{CP}^{2m+1}$ is classified by the multiple $k$ of an indivisible integral cohomology class in $H^2(\mathbb{CP}^{2m+1},\mathbb{Z})$. For our case in $\mathcal{B}_{FS}$, the principal orbit $G/K$ is the circle bundle $\mathrm{Prin}(1)$ over the K\"ahler--Einstein $\mathbb{CP}^{2m+1}$. This case is included in \cite{berard-bergery_sur_1982}.

The reduced system on the invariant set $$\mathcal{B}_{ALC}:=\mathcal{E}\cap\mathcal{P}\cap \{Y_1\equiv 0, X_1\equiv 0\}$$ carries two pieces of information. On one hand, if $a= O(1)$ while $b,c= O(t)$ at the infinity of some cohomogeneity one Einstein metrics, variables $Y_1$ and $X_1$ converge to zero along the corresponding integral curve.  Hence $\mathcal{B}_{ALC}$ serves as the ``invariant set of ALC limit''. On the other hand, the subsystem on $\mathcal{B}_{ALC}$ is essentially the one that appears in \cite{wink_cohomogeneity_2017}\cite{bohm_non-compact_1999} with respect to the group triple $(Sp(m)U(1),Sp(m)Sp(1),Sp(m+1))$. For $m=1$, there exists a $G_2$ metric on the cohomogeneity one space \cite{bryant_construction_1989}\cite{gibbons_einstein_1990}. The metric can be represented by a straight line in terms of variables in \eqref{new variable}.

Finally, for $m=1$, there exists a pair of invariant sets $\mathcal{B}^\pm_{\spin(7)}$ that represent the $\spin(7)$ conditions of positive/negative chirality. This case is studied in \cite{cvetic_new_2004} and a continuous 1-parameter family of $\spin(7)$ metrics is discovered. On one boundary of this family lies the $\spin(7)$ metric in \cite{bryant_construction_1989}\cite{gibbons_einstein_1990}. This case is discussed in more details in Section \ref{m=1}.

\section{Critical Points}
\label{Critical Points}
We study critical points of vector field $V_{\Lambda\leq 0}$ in \eqref{Polynomial Einstein Equation} in this section. Let $P$ be a critical point of $V_{\Lambda\leq 0}$. If an integral curve defined on $\mathbb{R}$ has $P$ as its limit as $\eta\to -\infty$, then the coordinates of $P$ represent the initial condition for the metric $dt^2+g_{G/K}(t)$ as $t\to 0$ up to the first order. Indeed, we see that initial conditions \eqref{initial condition G/H}, \eqref{initial condition Taub-NUT-1} and \eqref{initial condition Taub-NUT-2} are transformed to critical points. On the other hand, if the integral curve has $P$ as its limit as $\eta\to \infty$, then $P$ represents the asymptotic limit for the metric as $t\to \infty$ up to first order. A critical point can carry these two pieces of information simultaneously. 
 
Through computing linearizations at these points, we are able to prove the existence of Einstein metrics that are defined on a tubular neighbourhood around $G/H$ and a neighbourhood around the origin of $\mathbb{R}^{4m+4}$. The proof for the completeness of these metrics then boils down to showing that these integral curves are defined on $\mathbb{R}$.

On $\mathcal{B}_{RF}=\partial \mathcal{E}\cap \mathcal{P}$, where the function $W$ vanishes, we have the following critical points and boundary conditions.
\begin{enumerate}
\item $P_{0}:=\left(\frac{1}{3},\frac{1}{3},0,\sqrt{2},\frac{\sqrt{2}}{3},0\right)$\\

\item $P_{AC-i}:=\left(\frac{1}{4m+3},\frac{1}{4m+3},\frac{1}{4m+3},y_1,y_2,y_3\right), \quad i=1,2$
\begin{enumerate}
\item $P_{AC-1}: y_1=\sqrt{2},\quad 2y_2=y_3=\frac{2\sqrt{2}}{4m+3}$
\item $P_{AC-2}: y_1=\sqrt{2},\quad 2y_2=(2m+3)y_3=\frac{4m+6}{4m+3}\sqrt{\frac{4m+2}{(2m+3)^2+2m}}$
\end{enumerate}

\item $P_{ALC-i}:=\left(0,\frac{1}{4m+2},\frac{1}{4m+2},0,y_2,y_3\right), \quad i=1,2$
\begin{enumerate}
\item $P_{ALC-1}: 2y_2=y_3=\frac{1}{2m+1}\sqrt{\frac{4m+1}{2m+2}}$
\item $P_{ALC-2}: 2y_2=(m+1)y_3=\frac{m+1}{4m+2}\sqrt{\frac{8m+2}{(m+1)^2+m}}$
\end{enumerate}

\item $P_{ALC-0}:=\left(0,\frac{1}{2},0,0,\frac{\sqrt{2}}{4},0\right)$

\item $\left(0,-\frac{1}{2},\frac{1}{2m},0,0,\frac{1}{m}\sqrt{\frac{2-m}{2}}\right),\quad m\leq 2$

\item $(a,a,b,y_1,0,0), \quad y_1\neq 0,\quad 3a^2+4mb^2=3a+4mb=1$

\item $(x_1,x_2,x_3,0,0,0), \quad x_1^2+2x_2^2+4mx_3^2=x_1+2x_2+4mx_3=1$
\end{enumerate}

On $\mathrm{int}(\mathcal{E})\cap \mathcal{P}$, we have the following.
\begin{enumerate}
\item $P_{AH}(y_1)=\left(\frac{1}{4m+3},\frac{1}{4m+3},\frac{1}{4m+3},y_1,0,0\right), \quad y_1\geq 0,\quad  W=\sqrt{\frac{1}{-\Lambda(4m+3)}}$\\

\item $P_{QK}=\left(\frac{1}{2m+3},\frac{1}{2m+3},\frac{1}{4m+6},\sqrt{2},0,\frac{\sqrt{2}}{2m+3}\right),\quad W=\frac{1}{2m+3}\sqrt{\frac{m+3}{-\Lambda}}$\\

\item $\left(\frac{m+2}{4(m+1)^2+m+2},\frac{2m+2}{4(m+1)^2+m+2},\frac{m+1}{4(m+1)^2+m+2},0,0,\sqrt{\frac{2}{4(m+1)^2+m+2}}\right),\quad W=\sqrt{\frac{m+2}{-\Lambda(4(m+1)^2+m+2)}}$
\end{enumerate}

\begin{figure}[h!]
\begin{center}
\includegraphics[width=5in]{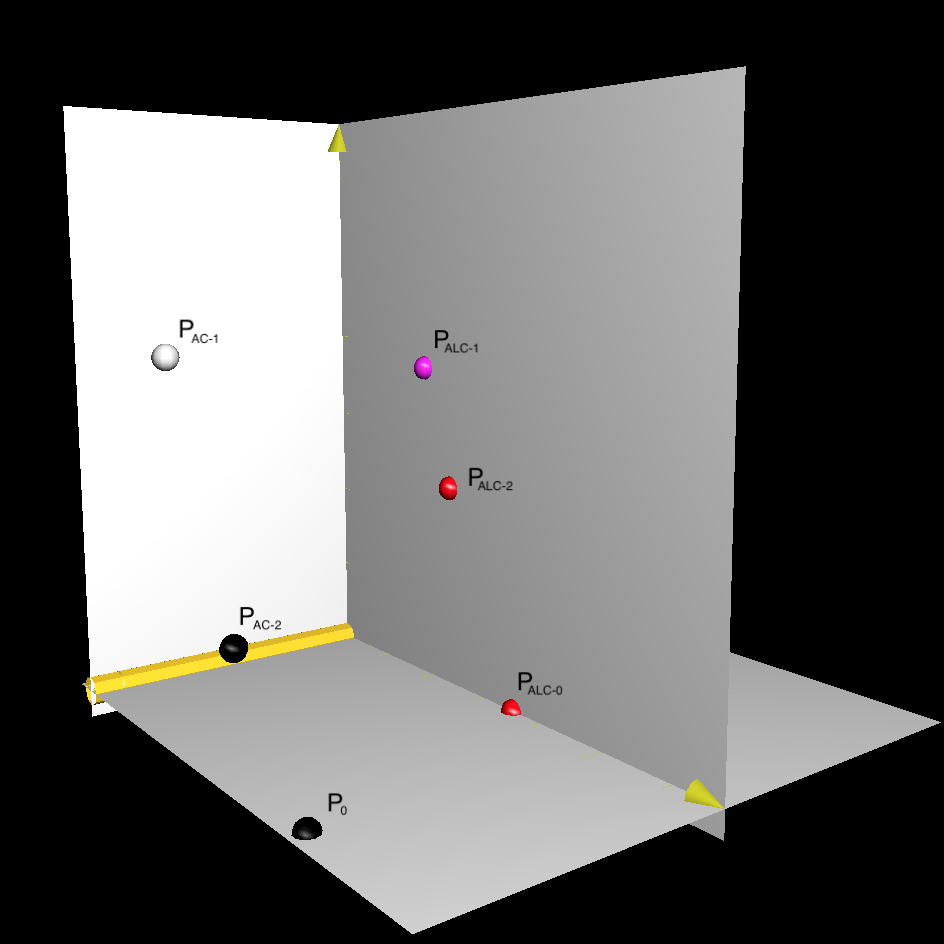}
\caption{Critical Points in $\mathcal{E}\cap \mathcal{P}$ Projected on $Y$-space}
\label{Fig1}
\end{center}
\end{figure}

In this article, we mainly focus on critical points $P_0$, $P_{AC-1}$, $P_{AC-2}$, $P_{ALC-2}$ and $P_{AH}(y_1)$. With the help of the software Maple, we compute the linearization $\mathcal{L}$ of \eqref{Polynomial Einstein Equation} at these critical points and compute the eigenvalues and eigenvectors. As we only consider system \eqref{Polynomial Einstein Equation} restricted on $\mathcal{E}$. We only focus on eigenvectors that are tangent to $\mathcal{E}$, i.e., orthogonal to $N_{\mathcal{E}}$ the normal vector field on $\mathcal{E}$.
Note that $\partial \mathcal{E}$ is the intersection of $\mathcal{E}$ and the algebraic surface $1-G-R_s=0$. Therefore, for integral curves that stay in $\partial \mathcal{E}$, eigenvectors are orthogonal to the normal vector field $N_{\partial \mathcal{E}}$ on the algebraic surface $1-G-R_s=0$ in addition to $N_{\mathcal{E}}$. We have
$$
\quad N_{\mathcal{E}}=\begin{bmatrix}
1\\
2\\
4m\\
0\\
0\\
0
\end{bmatrix},\quad 
\quad N_{\partial \mathcal{E}}=\begin{bmatrix}
2X_1\\
4X_2\\
8mX_3\\
-Y_1Y_2^2-\frac{m}{2}Y_1Y_3^2\\
-Y_1^2Y_2+8Y_2+4m(m+2)Y_3\\
-\frac{m}{2}Y_1^2Y_3-2mY_3+4m(m+2)Y_2
\end{bmatrix}.
$$

\subsection{$P_0$}
For an integral curve that emanates from $P_0=\left(\frac{1}{3},\frac{1}{3},0,\sqrt{2},\frac{\sqrt{2}}{3},0\right)$, one can show that the point is \eqref{initial condition G/H} under the new coordinate \eqref{Polynomial Einstein Equation}. Integral curves emanating from this point represent smooth Einstein metrics on the tubular neighbourhood of $G/H$.
The linearization at the point is
\begin{equation}
\mathcal{L}(P_0)=\begin{bmatrix}
-\frac{8m+6}{18m+9}&\frac{8m}{18m+9}&0&\frac{(12m+8)\sqrt{2}}{54m+27}&\frac{4m\sqrt{2}}{6m+3}&-\frac{4m(m+2)\sqrt{2}}{18m+9}\\
\frac{4m}{18m+9}&-\frac{4m+6}{18m+9}&0&-\frac{(12m+4)\sqrt{2}}{54m+27}&\frac{4m\sqrt{2}}{6m+3}&-\frac{4m(m+2)\sqrt{2}}{18m+9}\\
-\frac{1}{6m+3}&-\frac{2}{6m+3}&-\frac{2}{3}&\frac{\sqrt{2}}{18m+9}&-\frac{\sqrt{2}}{2m+1}&\frac{(m+2)\sqrt{2}}{6m+3}\\
\sqrt{2}&-\sqrt{2}&0&0&0&0\\
\frac{(4m+3)\sqrt{2}}{18m+9}&\frac{(2m+3)\sqrt{2}}{18m+9}&0&-\frac{2}{54m+27}&\frac{2}{6m+3}&\frac{4m(m+2)}{18m+9}\\
0&0&0&0&0&\frac{2}{3}
\end{bmatrix}
\end{equation}
Eigenvalues, along with their respective eigenvectors that are tangent to $\mathcal{E}$, are the following. 
$$
\quad \lambda_1=\lambda_2=\lambda_3=\frac{2}{3},\quad\lambda_4=-\frac{2}{3},\quad \lambda_5=-\frac{4}{3}
$$
\begin{equation}
\begin{split}
&v_1=\begin{bmatrix}
-4m(m+2)\\
-4m(m+2)\\
3(m+2)\\
0\\
-2\sqrt{2}m(m+2)\\
6\sqrt{2}
\end{bmatrix},
v_2=\begin{bmatrix}
-4\\
2\\
0\\
-9\sqrt{2}\\
-\sqrt{2}\\
0
\end{bmatrix},
v_3=\begin{bmatrix}
-4m\\
-4m\\
3\\
0\\
-2(m+1)\sqrt{2}\\
0
\end{bmatrix},
v_4=\begin{bmatrix}
-4m\sqrt{2}\\
-4m\sqrt{2}\\
3\sqrt{2}\\
0\\
4m\\
0
\end{bmatrix},
v_5=\begin{bmatrix}
-4\sqrt{2}\\
2\sqrt{2}\\
0\\
9\\
1\\
0
\end{bmatrix}
\end{split}
\end{equation}
Hence the general linearized solution emanating from $P_0$ is of the form
\begin{equation}
\label{linearized zeta pre}
P_0+s_1 e^\frac{2\eta}{3}v_1+s_2 e^\frac{2\eta}{3}v_2+s_3 e^\frac{2\eta}{3}v_3
\end{equation}
for some constants $s_i\in \mathbb{R}$. Note that the correspondence between germs of linearized solution \eqref{linearized zeta pre} and $(s_1,s_2,s_3)\in\mathbb{R}^3$ is not 1 to 1. For example, $(1,1,1)$ and $(2,2,2)$ give the same linearized solution. The redundancy is cut out by fixing $\sum_{i=1}^3s_i^2=1$.
By Hartman--Grobman theorem, there is a 1 to 1 correspondence between each choice of $(s_1,s_2,s_3)\in \mathbb{S}^2$ and an actual solution curve that emanates $P_0$. Hence we can use $\zeta_{(s_1,s_2,s_3)}$ to denote the actual solution that approaches to \eqref{linearized zeta pre} near $P_0$.  Moreover, by the unstable version of Theorem 4.5 in \cite{coddington_theory_1955}, there is some $\delta>0$ that
\begin{equation}
\label{linearized zeta negative einstein}
\zeta_{(s_1,s_2,s_3)} = P_0+s_1e^\frac{2\eta}{3}v_1+s_2e^\frac{2\eta}{3}v_2+s_3e^\frac{2\eta}{3}v_3+O\left(  e^{\left(\frac{2}{3}+\delta\right)\eta} \right).
\end{equation}

\begin{remark}
\label{different eigenvalues}
Here we explain the advantage of using system \eqref{Polynomial Einstein Equation} instead of \eqref{Polynomial Einstein Equation pre}. The linearization of \eqref{Polynomial Einstein Equation pre} at $P_0$ has two distinct positive eigenvalues. Hence the error term of a linearized solution may dominates terms with the smaller eigenvalues, which create extra difficulties in estimating a function near $P_0$. In \eqref{linearized zeta negative einstein}, we only have one positive eigenvalue. As the error of the linearized solution is dominated near $P_0$, we can safely make an estimate using the linearized solution.
\end{remark}

In this article, we consider $\zeta_{(s_1,s_2,s_3)}$ with $s_1>0$ and $s_2,s_3\geq 0$. In order the let $\zeta_{(s_1,s_2,s_3)}$ enter $\mathcal{E}\cap \mathcal{P}$ initially, we must have $s_1> 0$ so that $Y_3$ is positive initially along the curve. The geometric meaning of having $s_2\geq 0$ is to allow $H/K$ to be squashed in a way that $a^2\leq 2b^2$ for $dt^2+a^2\left.Q\right|_{\mathbb{I}}+b^2\left.Q\right|_{[\mathfrak{t}^2]_\mathbb{R}}+c^2\left.Q\right|_{\mathfrak{q}_-}$. Whether there exists a complete metric that is represented by $\zeta_{(s_1,s_2,s_3)}$ with $s_2<0$ is to be known. In order to let  $\zeta_{(s_1,s_2,s_3)}$ enter $\mathcal{E}\cap\mathcal{P}$ initially, we must have $s_3\geq 0$.

It is clear that $P_0\in\partial \mathcal{E}$. Since $N_{\partial\mathcal{E}}(P_0)$ is parallel to
$$
\begin{bmatrix}
3\\
6\\
0\\
-\sqrt{2}\\
9\sqrt{2}\\
6\sqrt{2}m(m+2)
\end{bmatrix},
$$
one can check that $v_1$ and $v_2$ are orthogonal to $N_{\partial\mathcal{E}}(P_0)$. Therefore, the 1-parameter family $\zeta_{(s_1,s_2,0)}$ stays in the invariant set $\mathcal{B}_{RF}$. Hence each $\zeta_{(s_1,s_2,0)}$ near $P_0$ in $\mathcal{E}\cap \mathcal{P}$ represents a Ricci-flat metric defined on the tubular neighborhood around $\mathbb{HP}^m$. Each $\zeta_{(s_1,s_2,s_3)}$ with $s_3>0$ near $P_0$ represents a negative Einstein metric defined on the tubular neighborhood around $\mathbb{HP}^m$.

There are some $\zeta_{(s_1,s_2,s_3)}$ known to be defined on $\mathbb{R}$. Note that $\zeta_{(s_1,0,s_3)}$ lies on $\mathcal{B}_{Rd}$. These integral curves are of two summands type. By \cite{wink_cohomogeneity_2017}\cite{bohm_non-compact_1999}, we know that each $\zeta_{(1,0,0)}$ is an integral curve on $\mathbb{R}$ that originates from $P_0$ and tends to $P_{AC-2}$ and each $\zeta_{(s_1,0,s_3)}$ with $s_3>0$ is an integral curves that originates from $P_0$ and tend to $P_{AH}(\sqrt{2})$. $\zeta_{(s_1,s_2,0)}$ with $s_2>0$ in the case $m=1$ were studied in \cite{cvetic_new_2004}. These integral curves all tend to $P_{ALC-2}$. In Section \ref{Compact invariant set}, we construct a compact invariant set that contains all $\zeta_{(s_1,s_2,s_3)}$ with $s_1,s_2,s_3\geq 0$.

\subsection{$P_{AC-1}$ and $P_{AC-2}$}
\label{AC1 and AC2}
Consider 
$P_{AC-1}=\left(\frac{1}{4m+3},\frac{1}{4m+3},\frac{1}{4m+3},\sqrt{2},\frac{\sqrt{2}}{4m+3},\frac{2\sqrt{2}}{4m+3}\right).$
It is clear that the point corresponds to the initial condition \eqref{initial condition Taub-NUT-1}. We have 
\tiny
\begin{equation}
\mathcal{L}(P_{AC-1})=\begin{bmatrix}
-\frac{4m+2}{4m+3}&0&0&\frac{8\sqrt{2}(2m+1)(m+1)}{(4m+3)^3}&-\frac{8\sqrt{2}m(m+1)}{(4m+3)^2}&\frac{4\sqrt{2}m(m+1)}{(4m+3)^2}\\
0&-\frac{4m+2}{4m+3}&0&-\frac{4\sqrt{2}(m+1)}{(4m+3)^3}&-\frac{8\sqrt{2}m(m+1)}{(4m+3)^2}&\frac{4\sqrt{2}m(m+1)}{(4m+3)^2}\\
0&0&-\frac{4m+2}{4m+3}&-\frac{4\sqrt{2}(m+1)}{(4m+3)^3}&\frac{6\sqrt{2}(m+1)}{(4m+3)^2}&-\frac{3\sqrt{2}(m+1)}{(4m+3)^2}\\
\sqrt{2}&-\sqrt{2}&0&0&0&0\\
\frac{\sqrt{2}}{(4m+3)(2m+1)}&-\frac{(2m-1)\sqrt{2}}{(4m+3)(2m+1)}&\frac{4\sqrt{2}m}{(4m+3)(2m+1)}&-\frac{2}{(4m+3)^3}&\frac{4m+6}{(4m+3)^2}&\frac{2m}{(4m+3)^2}\\
\frac{2\sqrt{2}}{(4m+3)(2m+1)}&\frac{(4m+6)\sqrt{2}}{(4m+3)(2m+1)}&-\frac{4\sqrt{2}}{(4m+3)(2m+1)}&-\frac{4}{(4m+3)^3}&\frac{8m+12}{(4m+3)^2}&\frac{4m}{(4m+3)^2}
\end{bmatrix}.
\end{equation}
\normalsize
Eigenvectors, along with their respective eigenvalues, that are tangent to $\mathcal{C}\cap\mathcal{H}$ are the following. 
$$
\lambda_1=\lambda_2=\lambda_3=\frac{2}{4m+3},\quad \lambda_4=\lambda_5=-\frac{4m+4}{4m+3}
$$
\begin{equation}
\begin{split}
&v_1=\begin{bmatrix}
-4m\sqrt{2}\\
-4m\sqrt{2}\\
3\sqrt{2}\\
0\\
4m\\
-(8m+12)
\end{bmatrix},
v_2=\begin{bmatrix}
-(4m+2)\sqrt{2}\\
\sqrt{2}\\
\sqrt{2}\\
-(4m+3)^2\\
-(4m+3)\\
-(8m+6)
\end{bmatrix},
v_3=\begin{bmatrix}
0\\
0\\
0\\
0\\
-1\\
-2
\end{bmatrix},\\
&v_4=\begin{bmatrix}
-4\sqrt{2}m(m+1)\\
0\\
\sqrt{2}(m+1)\\
2m(4m+3)\\
0\\
2
\end{bmatrix},
v_5=\begin{bmatrix}
-4\sqrt{2}(m+1)^2\\
2\sqrt{2}(m+1)\\
\sqrt{2}(m+1)\\
(4m+3)(2m+3)\\
1\\
0
\end{bmatrix}
\end{split}
\end{equation}
Therefore, there exists a 2-parameter family of integral curves $\gamma_{(s_1,s_2,s_3)}$ with $(s_1,s_2,s_3)\in\mathbb{S}^2$ that emanate from $P_{AC-1}$ such that
\begin{equation}
\label{linearized gamma negative einstein}
\gamma_{(s_1,s_2,s_3)}=P_{AC-1}+s_1e^\frac{2\eta}{4m+3}v_1+s_2e^\frac{2\eta}{4m+3}v_2+s_3e^\frac{2\eta}{4m+3}v_3+O\left( e^{\left(\frac{2}{4m+3}+\delta\right)\eta} \right).
\end{equation}

In this article, we consider $\gamma_{(s_1,s_2,s_3)}$ with $s_1,s_2,s_3\geq 0$.  The choice for $s_1\geq 0$ is to allow the $\mathbb{CP}^{2m+1}$ in $G/K$ to be squashed in a way that $b^2\leq 2c^2$ for $dt^2+a^2\left.Q\right|_{\mathbb{I}}+b^2\left.Q\right|_{[\mathfrak{t}^2]_\mathbb{R}}+c^2\left.Q\right|_{\mathfrak{q}_-}$. The geometric meaning of having $s_2\geq 0$ is the same as the one for $\zeta_{(s_1,s_2,s_3)}$. In order the let $\gamma_{(s_1,s_2,s_3)}$ enter $\mathcal{E}\cap \mathcal{P}$ initially, we must have $s_3\geq 0$. 

One can check that $P_{AC-1}\in \partial\mathcal{E}$. Since $N_{\partial\mathcal{E}}(P_{AC-1})$ is parallel to
$$
\begin{bmatrix}
4m+3\\
2(4m+3)\\
4m(4m+3)\\
-(2m+1)\sqrt{2}\\
(2m+3)(2m+1)\sqrt{2}\\
m(2m+1)\sqrt{2}\\
\end{bmatrix},
$$
it is clear that $\gamma_{(s_1,s_2,0)}$ is a 1-parameter family of integral curves that stay in $\mathcal{B}_{RF}$. Hence one obtain a 1-parameter family of Ricci-flat metrics and a 2-parameter family of negative Einstein metrics on the neighborhood around the origin in $\mathbb{R}^{4m+4}$.

Some $\gamma_{(s_1,s_2,s_3)}$ are known to be defined on $\mathbb{R}$. A trivial example is $\gamma_{(0,0,0)}$ that represent the standard Euclidean metric. With $s_1>0$ and $s_2\geq 0$, $\gamma_{(s_1,0,s_2)}$ stays in $\mathcal{B}_{Rd}$, with $\lim\limits_{\eta\to \infty}\gamma_{(1,0,0)}=P_{AC-2}$ and  $\lim\limits_{\eta\to \infty}\gamma_{(s_1,0,s_2)}=P_{AH}(\sqrt{2})$ for $s_2>0$ \cite{chi_cohomogeneity_2019}. Moreover, $\gamma_{(0,0,1)}$ is simply the hyperbolic cone with the standard sphere as its base. It is also known that  $\gamma_{(0,s_2,s_3)}$ stays in $\mathcal{B}_{FS}$.  In particular, $\gamma_{(0,1,0)}$ is the almost K\"ahler--Einstein metric with $P_{ALC-1}$ as  its limit \cite{berard-bergery_sur_1982}\cite[Theorem 9.130]{besse_einstein_2008}. For $s_2, s_3>0$, we know that $\lim\limits_{\eta\to \infty}\gamma_{(0,s_2,s_3)}=P_{AH}(y_1)$ for some $y_1\in [0,\sqrt{2})$. As shown in Section \ref{QK and HK}, there also exists an isolated example for another value of $(s_1,s_2,s_3)$, which is the quaternionic K\"ahler metric constructed in \cite{swann_hyper-kahler_1991}. 

As for $P_{AC-2}=\left(\frac{1}{4m+3},\frac{1}{4m+3},\frac{1}{4m+3},\sqrt{2},y_2,y_3\right)$, 
where $y_2=\frac{2m+3}{4m+3}\sqrt{\frac{4m+2}{(2m+3)^2+2m}}$ and $y_3=\frac{2}{4m+3}\sqrt{\frac{4m+2}{(2m+3)^2+2m}}$,
the point corresponds to initial condition \eqref{initial condition Taub-NUT-2}. Moreover, by Lemma 4.4 in \cite{chi_invariant_2019}, we know that if an integral curve defined on $\mathbb{R}$ converges to $P_{AC-2}$, then the Einstein metric represented has an AC asymptotic limit as
$$dt^2+\beta^2t^2\left(\left.Q\right|_{\mathbb{I}}+\frac{1}{2}\left.Q\right|_{[\mathfrak{t}^2]_\mathbb{R}}+\frac{2m+3}{4}\left.Q\right|_{\mathfrak{q}_-}\right),$$
where $(4m+3)(4m+2)\beta^2=6+\frac{16m(m+2)(2m+3)-12m}{(2m+3)^2}$. 

Eigenvalues of $\mathcal{L}(P_{AC-2})$, whose corresponding eigenvectors are tangent to $\mathcal{E}$, are 
$$
\lambda_1=\frac{2}{4m+3},\quad \rho_1,\quad \rho_2,\quad \sigma_1\quad \sigma_2,
$$
where $\rho_2<0<\frac{2}{4m+3}<\rho_1$ are two roots of
\begin{equation*}
\begin{split}
y&=(64m^4+320m^3+516m^2+342m+81)x^2\\
&\quad +(64m^4+304m^3+448m^2+264m+54)x\\
&\quad -(64m^3+240m^2+248m+72).
\end{split}
\end{equation*}
and $\sigma_2<\sigma_1<0$ are two roots of 
\begin{equation*}
\begin{split}
y&=(64m^4+320m^3+516m^2+342m+81)x^2\\
&\quad +(64m^4+304m^3+448m^2+264m+54)x\\
&\quad +(32m^3+96m^2+88m+24).
\end{split}
\end{equation*}
The eigenvectors that correspond to $\frac{2}{4m+3}$ and $\rho_1$ are respectively 
$$
v_1=\begin{bmatrix}
0\\
0\\
0\\
0\\
-(2m+3)\\
-2
\end{bmatrix}, v_2=\begin{bmatrix}
-2\rho_1\\
\rho_1\\
0\\
-3\sqrt{2}\\
-y_2\\
y_3
\end{bmatrix}.
$$
It is straightforward to check that $P_{AC-2}\in\partial \mathcal{E}$ and $v_2$ is orthogonal to $N_{\partial \mathcal{E}}(P_{AC-2})$. Therefore there exists an integral curve $\Gamma$ on $\partial\mathcal{E}$ such that
$$
\Gamma=P_{AC-2}+e^{\rho_1\eta} v_2+O\left( e^{(\rho_1+\delta)\eta}\right).
$$
On the other hand, it is easy to check that $P_{AC-2}+e^{\frac{2\eta}{4m+3}}v_1$ is the hyperbolic cone with Jensen sphere as its base. In fact, the critical point is actually a sink in the subsystem restricted on $\mathcal{B}_{Rd}\cap\mathcal{B}_{RF}$ and $v_1$ is the only unstable eigenvector for $P_{AC-2}$ in the subsystem $\mathcal{B}_{Rd}$. In order to obtain new integral curves, we consider linearized solution in the form of
$$
P_{AC-2}+e^{\frac{2\eta}{4m+3}} v_1+s e^{\rho_1\eta} v_2
$$
for some $s\in\mathbb{R}$. If some actual solution $\Gamma_{s}$ corresponds to the linearized solution with $s \neq 0$, then as discussed in Remark \ref{different eigenvalues}, we have
$$
\Gamma_s=P_{AC-2}+ e^{\frac{2\eta}{4m+3}} v_1+s e^{\rho_1\eta} v_2+O\left( e^{\left(\frac{2}{4m+3}+\delta\right)\eta}\right)
$$
for some $\delta>0$. However, the third term can possibly be merged into $O\left( e^{\left(\frac{2}{4m+3}+\delta\right)\eta}\right)$ since it is possible that $\frac{2}{4m+3}+\delta<\rho_1$. In that way, the value of $s$ is difficult to trace.

\subsection{$P_{ALC-2}$ and  $P_{AH}(y_1)$}
Einstein metrics constructed in this article are represented by integral curves that emanate from $P_0$, $P_{AC-1}$ and $P_{AC-2}$. In Section \ref{Asymptotic}, we show that most of the integral curves of Ricci-flat metrics converges to $P_{ALC-2}$. 

Recall that
$P_{ALC-2}=\left(0,\frac{1}{4m+2},\frac{1}{4m+2},0,\frac{m+1}{8m+4}\sqrt{\frac{8m+2}{(m+1)^2+m}},\frac{1}{4m+2}\sqrt{\frac{8m+2}{(m+1)^2+m}}\right).$ We claim the following.
\begin{proposition}
\label{PALC is ALC limit}
If an integral curve defined on $\mathbb{R}$ converges to $P_{ALC-2}$, then the Einstein metric represented is ALC.
\end{proposition}
\begin{proof}
By the assumption, we have 
$$
\lim_{t\to\infty}\dot{b}=\lim_{\eta\to \infty}\frac{X_2}{Y_2}=\frac{2}{m+1}\sqrt{\frac{(m+1)^2+m}{8m+2}},\quad \lim_{t\to\infty}\dot{c}=\lim_{\eta\to \infty}\frac{X_3}{Y_3}=\sqrt{\frac{(m+1)^2+m}{8m+2}},
$$
$$
\lim_{t\to\infty}\frac{\dot{a}}{\dot{b}}=\lim_{\eta\to \infty}\frac{X_1Y_1}{X_2}=0
$$
Hence it is necessary that $\lim\limits_{t\to\infty}\dot{a}=0$. The metric represented has asymptotic limit as
$$
dt^2+C\left.Q\right|_{\mathbb{I}}+t^2\left(  \frac{2((m+1)^2+m)}{(m+1)^2(4m+1)}  \left.Q\right|_{[\mathfrak{t}^2]_\mathbb{R}}+  \frac{(m+1)^2+m}{8m+2}\left.Q\right|_{\mathfrak{q}_-}\right)
$$
for some constant $C>0$
\end{proof}
\begin{proposition}
$P_{ALC-2}$ is a sink in $\left(\partial \mathcal{E},\left.V_{\Lambda \leq 0}\right|_{\mathcal{B}_{RF}}\right)$
\end{proposition}
\begin{proof}
We prove the proposition by computing the linearization of \eqref{Polynomial Einstein Equation} at $P_{ALC-2}$ and then show that all unstable eigenvectors are not tangent to $\mathcal{E}$. Let $\alpha=\sqrt{\frac{8m+2}{(m+1)^2+m}}$. The linearization of \eqref{Polynomial Einstein Equation pre} at this point is
\begin{equation}
\begin{split}
&\mathcal{L}(P_{ALC-2})\\
&=\begin{bmatrix}
-\frac{4m+1}{4m+2}&-\frac{1}{(2m+1)^2}&-\frac{2m}{(2m+1)^2}&0&-\frac{(m^2+3m+1)\alpha}{(2m+1)^2}&-\frac{(m^2+3m+1)m\alpha}{2(2m+1)^2}\\
0&-\frac{8m^3+10m^2+4m}{(2m+1)^3}&\frac{m}{(2m+1)^3}&0&\frac{(4m^3+3m^2+3m+1)\alpha}{2(2m+1)^3}&-\frac{(4m^4+5m^3-m^2-m)\alpha}{4(2m+1)^3}\\
0&\frac{1}{2(2m+1)^3}&-\frac{16m^3+20m^2+6m+1}{2(2m+1)^3}&0&-\frac{(m^2-2m-1)\alpha}{2(2m+1)^3}&\frac{(3m^3+6m^2+2m)\alpha}{4(2m+1)^3}\\
0&0&0&-\frac{1}{4m+2}&0&0\\
0&-\frac{(m+1)(2m^2-1)\alpha}{2(2m+1)^3}&\frac{(4m^3+7m^2+3m)\alpha}{2(2m+1)^3}&0&\frac{4m^2+5m+1}{2(2m+1)^3}&\frac{4m^3+5m^2+m}{4(2m+1)^3}\\
0&\frac{2(m+1)^2}{(2m+1)^3}&-\frac{(m+1)\alpha}{(2m+1)^3}&0&\frac{4m+1}{(2m+1)^3}&\frac{4m^2+m}{2(2m+1)^3}\\
\end{bmatrix}.
\end{split}
\end{equation}
Eigenvalues are the following. 
$$
\lambda_1=-\frac{1}{4m+2},\quad \lambda_2=\lambda_3=-\frac{4m+1}{4m+2},\quad \lambda_4=\rho_1,\quad \lambda_5=\rho_2,\quad \lambda_6=\frac{1}{2m+1}
$$
where $\rho_1<\rho_2<0$ are roots of 
$$
y=(8m^4+32m^3+34m^2+14m+2)x^2+(8m^4+30m^3+27m^2+9m+1)x+(4m^3+5m^2+m).
$$
Since $\mathcal{B}_{RF}$ is a 4-dimensional invariant set, four of the eigenvectors must be tangent to $\mathcal{B}_{RF}$. Since $\lambda_6$ is the only non-negative eigenvalue, in order to show that $P_{ALC-2}$ is a sink in $\left(\partial \mathcal{E},\left.V_{\Lambda \leq 0}\right|_{\mathcal{B}_{RF}}\right)$, it is sufficient to show that the eigenvector corresponds to $\lambda_6$ is not tangent to $\mathcal{B}_{RF}$. Indeed, computation shows that the eigenvector corresponds to $\lambda_6$ and normal vector field of $\partial \mathcal{E}$ at $P_{ALC-2}$ are  
$$v_6=
\begin{bmatrix}
-(4m+2)\sqrt{(m+1)^2+m}\\
\sqrt{(m+1)^2+m}\\
\sqrt{(m+1)^2+m}\\
0\\
(m+1)^2\sqrt{8m+2}\\
(2m+2)\sqrt{8m+2}
\end{bmatrix},\quad N_{\partial\mathcal{E}}(P_{ALC-2})=\begin{bmatrix}
0\\
\frac{2}{2m+1}\\
\frac{4m}{2m+1}\\
0\\
\frac{2}{2m+1}\sqrt{((m+1)^2+m)(8m+2)}\\
\frac{m}{2m+1}\sqrt{((m+1)^2+m)(8m+2)}\\
\end{bmatrix},
$$
which are not orthogonal. 
Hence the vector is not tangent to $\mathcal{\partial E}$. The proof is complete.
\end{proof}

It is straightforward to verify that the set of all $P_{AH}(y_1)=\left(\frac{1}{4m+3},\frac{1}{4m+3},\frac{1}{4m+3},y_1,0,0\right)$ is a 1-dimensional invariant set in the interior of $\mathcal{E}$.
For any fix $y_1$, we have 
\begin{equation}
\mathcal{L}(P_{AH}(y_1))=\begin{bmatrix}
-1&0&0&0&0&0\\
0&-1&0&0&0&0\\
0&0&-1&0&0&0\\
y_1&-y_1&0&0&0&0\\
0&0&0&0&-\frac{1}{4m+3}&0\\
0&0&0&0&0&-\frac{1}{4m+3}
\end{bmatrix}
\end{equation}
Eigenvectors, along with their respective eigenvalues, that are tangent to $\mathcal{C}\cap\mathcal{H}$ are the following. 
$$
\lambda_1=0,\quad \lambda_2=\lambda_3=-\frac{1}{4m+3},\quad\lambda_4=\lambda_5=-1
$$
$$
v_1=\begin{bmatrix}
0\\
0\\
0\\
1\\
0\\
0
\end{bmatrix},
v_2=\begin{bmatrix}
0\\
0\\
0\\
0\\
1\\
0\\
\end{bmatrix},
v_3=\begin{bmatrix}
0\\
0\\
0\\
0\\
0\\
1\\
\end{bmatrix},
v_4=\begin{bmatrix}
-2\\
1\\
0\\
3y_1\\
0\\
0
\end{bmatrix},
v_5=\begin{bmatrix}
-4m\\
-4m\\
3\\
0\\
0\\
0
\end{bmatrix}
$$
Therefore, $P_{AH}:=\{P_{AH}(y_1)\mid y_1\geq 0\}$ is a 1-dimensional invariant stable manifold.

We say a critical point $P$ is a $(p,q)$-saddle if $P$ has unstable direction of dimension $p$ and stable direction of dimension $q$.
In summary, we have the following lemma.
\begin{lemma}
\label{Critical in W=0}
In the subsystem of \eqref{Polynomial Einstein Equation} restricted on $\mathcal{B}_{RF}=\partial \mathcal{E}$:
\begin{enumerate}
\item $P_0$ is a $(2,2)$-saddle.
\item $P_{AC-1}$ is a $(2,2)$-saddle. $P_{AC-2}$ is a $(1,3)$-saddle.
\item $P_{ALC-2}$ is a sink.
\end{enumerate}
\end{lemma}

\begin{lemma}
\label{Critical in W neq 0}
In system of \eqref{Polynomial Einstein Equation} on $\mathcal{E}$:
\begin{enumerate}
\item $P_0$ is a $(3,2)$-saddle.
\item $P_{AC-1}$ is a $(3,2)$-saddle. $P_{AC-2}$ is a $(2,3)$-saddle.
\item $P_{ALC-2}$ is a $(1,4)$-saddle.
\item $P_{AH}$ is a 1-dimensional stable manifold.
\end{enumerate}
\end{lemma}

\begin{remark}
It is worth mentioning that linearizations at $P_0$, $P_{AC-1}$ and $P_{AC-2}$ can be carried over to the compact case where $\Lambda>0$. The short existing integral curves correspond to positive Einstein metrics on the tubular neighborhood around $\mathbb{HP}^m$ or the origin of $\mathbb{R}^{4m+4}$. In \cite{page_einstein_1986}, numerical analysis indicates that there exists an inhomogeneous Einstein metric on $\mathbb{HP}^{m+1}\# \overline{\mathbb{HP}^{m+1}}$. If such a metric does exist, its restriction on the neighborhood around $\mathbb{HP}^m$ is represented by some integral curve that emanates from $P_0$. For a cohomogeneity one Einstein metric on $\mathbb{HP}^{m+1}\# \overline{\mathbb{HP}^{m+1}}$, the trace of the shape operator is supposed to vanish at some $t_*>0$. Therefore, one may need some other coordinate change in order to construct positive cohomogeneity one Einstein metrics.
\end{remark}

\section{Compact Invariant Set}
\label{Compact invariant set}
This section is dedicated to constructing a compact invariant set that contains critical points studied above in its boundary.
\begin{proposition}
\label{Set construction 1}
Let 
$$\mathcal{A}_1=\left\{(X_1,X_2,X_3,Y_1,Y_2,Y_3)\mid X_1- X_2\leq 0, \quad Y_1^2\leq 2\right\}$$
The set
$\mathcal{E}\cap \mathcal{P} \cap\mathcal{A}_1$ is flow-invariant.
\end{proposition}
\begin{proof}
Computation shows that
\begin{equation}
\begin{split}
\langle\nabla(Y_1^2),V_{\leq 0}\rangle\mid_{Y_1^2=2}&=2Y_1^2(X_1-X_2)\leq 0
\end{split}
\end{equation}
in $\mathcal{E}\cap\mathcal{A}_1$.
Moreover, we have
\begin{equation}
\label{X_1-X_2}
\begin{split}
&\langle\nabla(X_1-X_2),V_{\leq 0}\rangle\mid_{X_1-X_2=0}\\
&=(X_1-X_2)(G+\Lambda W^2-1)+\frac{m}{4}Y_3^2(Y_1^2-2)+Y_2^2(Y_1^2-2)\\
&\leq 0
\end{split}
\end{equation}
in $\mathcal{E}\cap\mathcal{A}_1$. The proof is complete.
\end{proof}

Define
\begin{equation}
\begin{split}
&\mathcal{A}_2\\
&=\left\{(X_1,X_2,X_3,Y_1,Y_2,Y_3)\mid 2Y_2-Y_3\geq 0, \frac{\sqrt{2}
}{2}(2Y_2-Y_3)+X_3-X_2\geq 0, X_2\leq \frac{1}{2}, X_3\geq 0 \right\}.
\end{split}
\end{equation}
We want to show that the set
$\mathcal{S}:=\mathcal{E}\cap\mathcal{P}\cap \mathcal{A}_1\cap\mathcal{A}_2$ is a flow-invariant compact set. We prove the compactness first.
\begin{figure}[h!]
\centering
\begin{subfigure}{.3\textwidth}
  \centering
  \includegraphics[width=1\linewidth]{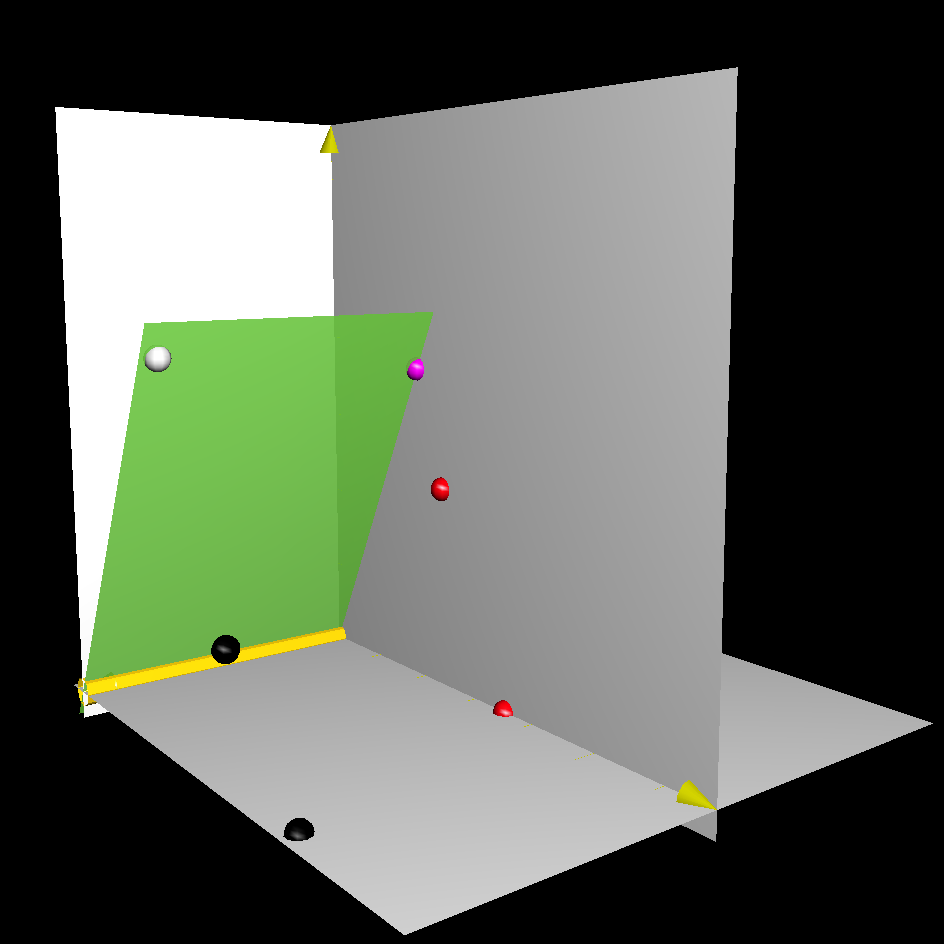}
  \caption{$2Y_2-Y_3\geq 0$}
  \label{mathcalS1}
\end{subfigure}
\begin{subfigure}{.3\textwidth}
  \centering
  \includegraphics[width=1\linewidth]{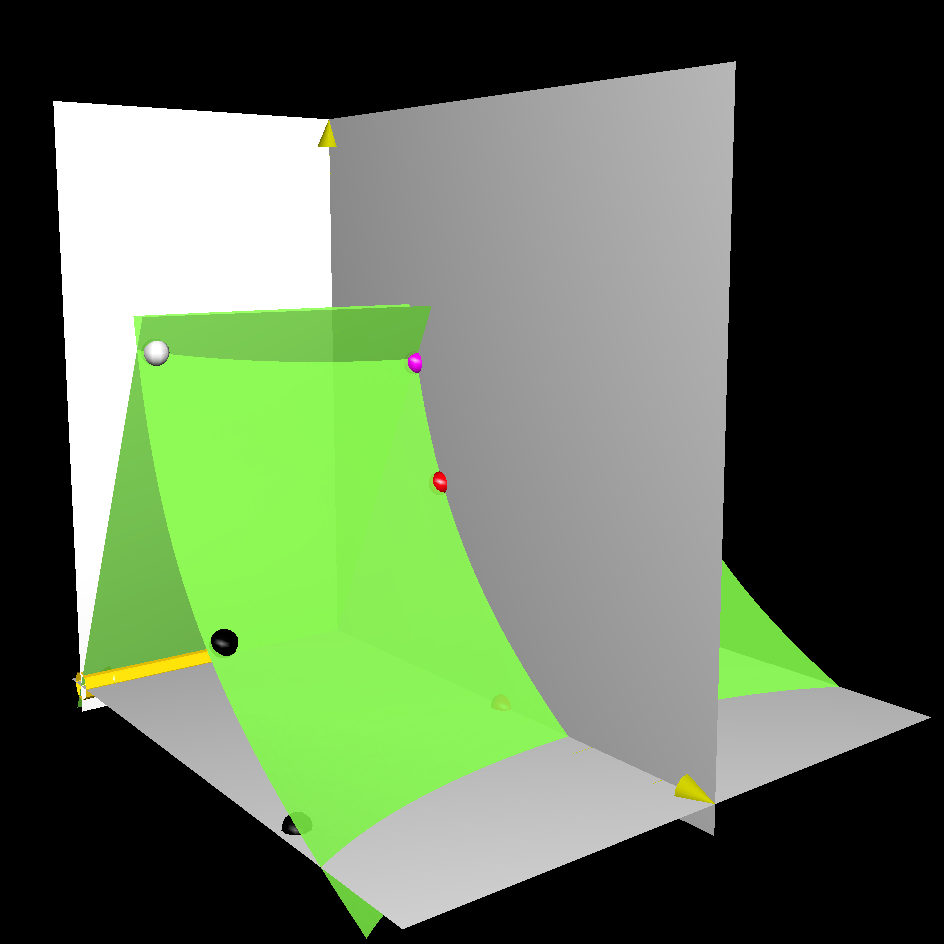}
  \caption{$G\geq \frac{1}{4m+3}$}
  \label{fmathcalS2}
\end{subfigure}
\begin{subfigure}{.3\textwidth}
  \centering
  \includegraphics[width=1\linewidth]{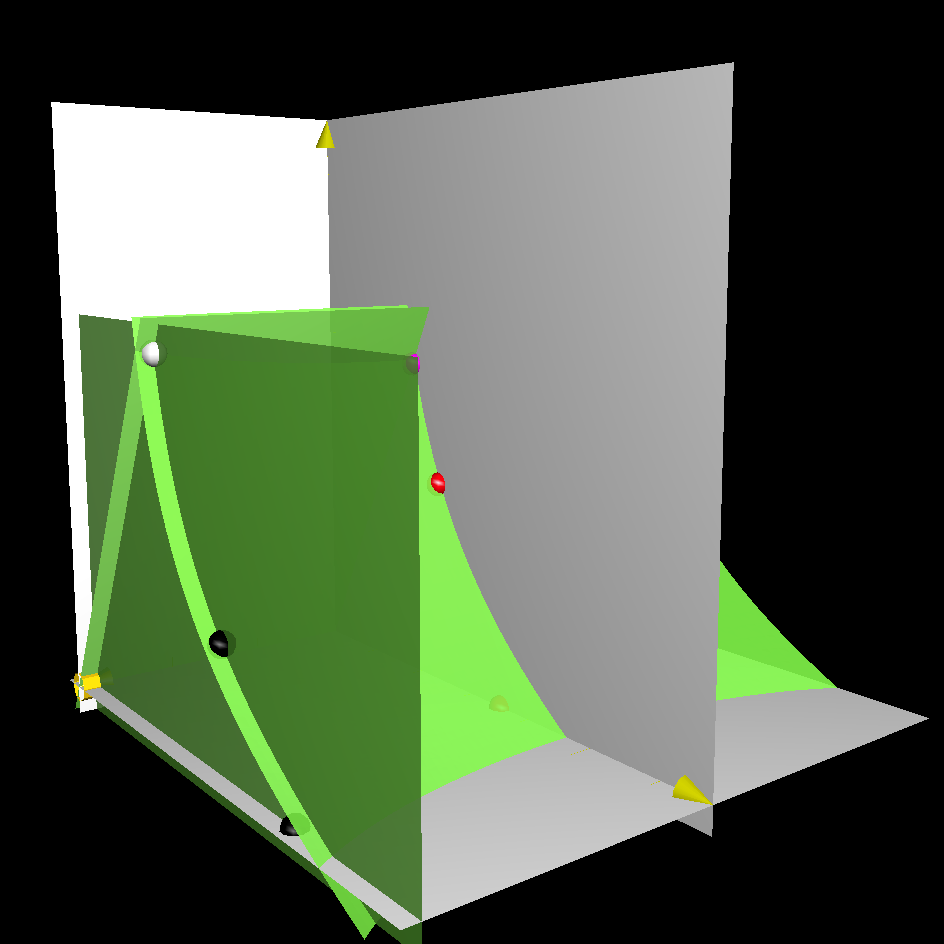}
  \caption{$Y_1\leq \sqrt{2}$}
  \label{mathcalS3}
\end{subfigure}
\caption{``Picture proof'' of the Compactness of $\mathcal{S}$}
\label{mathcalS}
\end{figure}

\begin{proposition}
The set $\mathcal{S}$ is compact.
\end{proposition}
\begin{proof}
From \eqref{Polynomial Conservation Law}, it is clear that the compactness is proven once we can show that $Y_i$'s are bounded above. By the definition of $\mathcal{A}_1$, we know that $Y_1$ is bounded above. By the definition of $\mathcal{A}_2$, we know that $Y_3$ is bounded above by $2Y_2$. From the definition of $\mathcal{E}$, we have
\begin{equation}
\label{upper bound Y2}
\begin{split}
1&\geq G+R_s\\
&= G+4Y_2^2-mY_3^2 +4m(m+2)Y_2Y_3-\frac{1}{2}Y_1^2Y_2^2-\frac{m}{4}Y_1^2Y_3^2\\
&\geq  G+3Y_2^2-mY_3^2 +4m(m+2)Y_2Y_3-\frac{m}{2}Y_3^2\quad \text{since $Y_1^2\leq 2$}\\
&\geq \frac{1}{4m+3}+3Y_2^2+\left(2m^2+\frac{5m}{2}\right)Y_3^2\quad \text{since $2Y_2\geq Y_3$}\\
&\geq \frac{1}{4m+3}+3Y_2^2
\end{split}
\end{equation}
Hence $Y_2^2<\frac{1}{3}$. The proof is complete. An illustration of the projection of $\mathcal{S}$ on $Y$-space is given in Figure \ref{mathcalS}.
\end{proof}

Before we prove that $\mathcal{S}$ is flow-invariant, we need to prove the following technical proposition.
\begin{proposition}
\label{technical comparison}
If $\frac{\sqrt{2}}{2}(2Y_2-Y_3)+X_3-X_2=0$ on $\mathcal{S}$, then 
$$
\frac{\sqrt{2}}{4}+\frac{m-1}{2}Y_3-Y_2+\frac{1}{8}Y_1^2\left(2Y_2+Y_3\right)\geq 0$$ on $\mathcal{S}.$
\end{proposition}
\begin{proof}
If $\frac{\sqrt{2}}{2}(2Y_2-Y_3)+X_3-X_2=0$, the by \eqref{Polynomial Conservation Law}, we have
\begin{equation}
\begin{split}
1&=X_1^2+2X_2^2+4mX_3^2-(4m+2)\Lambda W^2\\
&\quad +4Y_2^2-mY_3^2 +4m(m+2)Y_2Y_3-\frac{1}{2}Y_1^2Y_2^2-\frac{m}{4}Y_1^2Y_3^2\\
&\geq 2\left(X_3+\frac{\sqrt{2}}{2}(2Y_2-Y_3)\right)^2\\
&\quad +4Y_2^2-mY_3^2 +4m(m+2)Y_2Y_3-\frac{1}{2}Y_1^2Y_2^2-\frac{m}{4}Y_1^2Y_3^2.
\end{split}
\end{equation}
Since $X_3\geq 0$ and $2Y_2-Y_3\geq 0$ in $\mathcal{S}$, we can drop terms with $X_3$ above. The computation continues as
\begin{equation}
\begin{split}
1&\geq (2Y_2-Y_3)^2\\
&\quad+4Y_2^2-mY_3^2 +4m(m+2)Y_2Y_3-\frac{1}{2}Y_1^2Y_2^2-\frac{m}{4}Y_1^2Y_3^2\\
&=\left(8-\frac{1}{2}Y_1^2\right)Y_2^2+\left(1-m-\frac{m}{4}Y_1^2\right)Y_3^2+(4m(m+2)-4)Y_2Y_3\\
&\geq \left(8-\frac{1}{2}Y_1^2\right)Y_2^2+\left(2m^2+\frac{5}{2}m-1\right)Y_3^2\quad   \text{since $2Y_2-Y_3\geq 0$ and $Y_1^2\leq 2$}\\
&\geq  \left(8-\frac{1}{2}Y_1^2\right)Y_2^2 \quad  \text{as coefficient for $Y_3^2$ is positive}
\end{split}.
\end{equation}
Since $Y_1^2\leq 2$, we know that 
$$Y_2^2\leq\frac{1}{8-\frac{1}{2}Y_1^2}$$ in $\mathcal{S}$ if $\frac{\sqrt{2}}{2}(2Y_2-Y_3)+X_3-X_2=0$ holds. Moreover, the inequality above implies
$$
\left(\frac{\sqrt{2}}{4}\frac{1}{1-\frac{1}{4}Y_1^2}\right)^2\geq \frac{1}{8-\frac{1}{2}Y_1^2}\geq Y_2^2
$$
as $Y_1^2\leq 2.$
Hence 
$$
\frac{\sqrt{2}}{4}-Y_2+\frac{1}{4}Y_1^2Y_2\geq 0
$$
Therefore.
$$\frac{\sqrt{2}}{4}+\frac{m-1}{2}Y_3-Y_2+\frac{1}{8}Y_1^2\left(2Y_2+Y_3\right)\geq 0
$$
on $\mathcal{S}$.
\end{proof}

\begin{lemma}
The compact set $\mathcal{S}$ is flow-invariant.
\end{lemma}
\begin{proof}
We have two check three inequalities in $\mathcal{A}_2$. Firstly, we have
\begin{equation}
\begin{split}
\langle\nabla(X_3),V_{\leq 0}\rangle\mid_{X_3=0}&=(m+2)Y_2Y_3-\frac{1}{8}Y_1^2Y_3^2-\frac{1}{2}Y_3^2-\Lambda W^2\\
&\geq \frac{m+2}{2}Y_3^2-\frac{3}{4} Y_3^2\\
&\geq 0
\end{split}.
\end{equation}

Note that $X_2\leq \frac{1}{2}$ is equivalent to $X_1+4mX_3\geq 0$ in $\mathcal{C}\cap \mathcal{H}$. We have
\begin{equation}
\begin{split}
&\langle\nabla(X_1+4mX_3),V_{\leq 0}\rangle\mid_{X_1+4mX_3=0}\\
&=\frac{1}{2}Y_1^2Y_2^2+\frac{m}{4}Y_1^2Y_3^2+4m\left((m+2)Y_2Y_3-\frac{1}{8}Y_1^2Y_3^2-\frac{1}{2}Y_3^2\right)-(1+4m)\Lambda W^2\\
&\geq 0
\end{split}.
\end{equation}

As for inequalities concerning $Y_i$'s, we have
\begin{equation}
\begin{split}
\langle\nabla(2Y_2-Y_3),V_{\leq 0}\rangle\mid_{2Y_2-Y_3=0}&=2Y_3(X_3-X_2)\\
&\geq \sqrt{2}Y_3(Y_3-2Y_2)\\
&= 0
\end{split}.
\end{equation}

Finally, we have
\begin{equation}
\begin{split}
&\left.\left\langle\nabla\left(\frac{\sqrt{2}}{2}(2Y_2-Y_3)+X_3-X_2\right),V\right\rangle\right|_{\frac{\sqrt{2}}{2}(2Y_2-Y_3)+X_3-X_2=0}\\
&=\left(\frac{\sqrt{2}}{2}(2Y_2-Y_3)+X_3-X_2\right)(G+\Lambda W^2-1)+\sqrt{2}Y_2(1-X_2)-\frac{\sqrt{2}}{2}Y_3(1+X_2-2X_3)\\
&\quad +(2Y_2-Y_3)\left(\frac{1}{8}Y_1^2\left(2Y_2+Y_3\right)+\frac{m+1}{2}Y_3-Y_2\right)\\
&= (2Y_2-Y_3)\left(\frac{\sqrt{2}}{2}-\frac{\sqrt{2}}{2}X_2+\frac{m-1}{2}Y_3-Y_2+\frac{1}{8}Y_1^2(2Y_2+Y_3)\right)\\
&\quad \text{on replacing all $X_3$ with $X_2+\frac{\sqrt{2}}{2}(Y_3-2Y_2)$}\\
&\geq (2Y_2-Y_3)\left(\frac{\sqrt{2}}{4}+\frac{m-1}{2}Y_3-Y_2+\frac{1}{8}Y_1^2(2Y_2+Y_3)\right) \quad \text{since $X_2\leq \frac{1}{2}$ in $\mathcal{S}$}
\end{split}.
\end{equation}
By Proposition \ref{technical comparison}, the computation result above is non-negative. The proof is complete.
\end{proof}

By looking into the linearization of \eqref{Polynomial Einstein Equation} at $P_0$, $P_{AC-1}$ and $P_{AC-2}$ in Section \ref{Critical Points}. We learn that $\zeta_{(s_1,s_2,s_3)}$ is in $\mathcal{S}$ initially for $s_1,s_2,s_3\geq 0$; $\gamma_{(s_1,s_2,s_3)}$ is in $\mathcal{S}$ initially for $s_1,s_2,s_3\geq 0$; $\Gamma_s$ is in $\mathcal{S}$ initially for $s\in[0,\epsilon)$ for some $\epsilon>0$. Therefore, all these integral curves are defined on $\mathbb{R}$. It is clear that $R_1, R_2, R_3\geq 0$ in $\mathcal{S}$. Hence by Proposition \ref{big eta means big t}, we obtain the following lemma, using the same notation for the integral curve and the metric represented.
\begin{lemma}
\label{completeness}
The following metrics are complete.
\begin{enumerate}
\item
Smooth metrics $\zeta_{(s_1,s_2,s_3)}$, $s_1>0$, $s_2,s_3\geq 0$ defined on $M$;
\item
Smooth metrics $\gamma_{(s_1,s_2,s_3)}$, $s_1,s_2,s_3 \geq 0$ defined on $\mathbb{R}^{4m+4}$;
\item
Singular metrics $\Gamma_s$ with $s\in [0,\epsilon)$ defined on $\mathbb{R}^{4m+4}.$
\end{enumerate}
\end{lemma}

\section{Asymptotics}
\label{Asymptotic}
We divide this section into two parts. We first study the asymptotics for the Ricci-flat metrics obtained in Theorem \ref{zeta curve}-\ref{Gamma curve}. Then we study the asymptotics for the negative Einstein metrics. Without further specifying, we use $\Theta$ to denote any of the Einstein metrics in Lemma \ref{completeness}. A general property for a $\Theta$ is the following.
\begin{proposition}
\label{nonnegative X}
All $X_i$'s are positive along each $\Theta$.
\end{proposition}
\begin{proof}
By the definition of $\mathcal{S}$, we know that $X_3> 0$ along each of the integral curves. It is also clear that $R_i$'s are non-negative in $\mathcal{S}$. Suppose $X_2$ reaches zero for some $\eta_*\in\mathbb{R}$ along $\Theta$. Then at that point we have
$$\left.\frac{d}{d\eta}\right|_{\eta=\eta_*}X_2(\Theta(\eta))=(X_2(G+\Lambda W^2-1)+R_2-\Lambda W^2)(\Theta(\eta_*))\geq R_2(\Theta(\eta_*))\geq 0,$$
a contradiction. Similar argument can be used to prove that $X_1$ must be positive along $\Theta$.
\end{proof}

\subsection{Asymptotics for Ricci-flat Metrics}
All discussion in this section is restricted on $\mathcal{B}_{RF}$, where the function $W$ vanishes. In the case $m=1$, the asymptotic limit for $\gamma_{(s_1,s_2,0)}$ was rigorously proven to be ALC by \cite{bazaikin_new_2007}. In this section, we provide another proof and generalize the result for $m\geq 1$.

\begin{proposition}
\label{X1Y1 vanish}
Let $\Theta$ be any of $\zeta_{(s_1,s_2,0)}$ with $s_2>0$, $\gamma_{(s_1,s_2,0)}$ with $s_2>0$ or $\Gamma$ in Theorem \ref{zeta curve}-\ref{Gamma curve}, we have 
$\lim\limits_{\eta\to\infty}Y_1(\Theta(\eta))=0$ and $\lim\limits_{\eta\to\infty}X_1(\Theta(\eta))=0$.
\end{proposition}
\begin{proof}
Since $Y_1'=Y_1(X_1-X_2)< 0$
along each of the integral curves, we know that $Y_1$ decreases to some $l\in[0,\sqrt{2})$ along $\Theta$. Suppose $l\neq 0$, then there exists some sequence $\{\eta_k\}_{k=1}^\infty$ with $\lim\limits_{k\to \infty}\eta_k=\infty$ that $\lim\limits_{k\to \infty}(X_2-X_1)(\Theta(\eta_k))=0$. 

On the other hand, we claim that there exists some $\delta>0$ such that $R_2-R_1\geq \delta$ along $\Theta$. Suppose not, then there exists some sequence $\{\tilde{\eta}_k\}_{k=1}^\infty$ with $\lim\limits_{k\to \infty}\tilde{\eta}_k=\infty$ such that 
$$
\lim\limits_{k\to \infty}(R_2-R_1)(\Theta(\tilde{\eta}_k))=\lim\limits_{k\to \infty}\left[(2-Y_1^2)\left(Y_2^2+\frac{m}{4}Y_3^2\right)\right](\Theta(\tilde{\eta}_k))=0.
$$
Therefore, for the same sequence $\{\tilde{\eta}_k\}_{k=1}^\infty$, it is necessary that
$$
\lim\limits_{k\to \infty}Y_2(\Theta(\tilde{\eta}_k))=\lim\limits_{k\to \infty}Y_3(\Theta(\tilde{\eta}_k))=0.
$$
Since $1-G-R_s=0$ on $\mathcal{B}_{RF}$, we conclude that there exists a point in the $\omega$-limit set of $\Theta$ of the form $(a,a,b,y_1,0,0)$, with $y_1\neq 0$ and $3a^2+4mb^2=3a+4mb=1$. Such a point is a critical point of type 6 as in Section \ref{Critical Points}. It is clear that one of $a$ and $b$ must be negative, a contradiction to Proposition \ref{nonnegative X}. 

Observe \eqref{X_1-X_2}, we can find a small enough $\epsilon>0$ such that $X_2-X_1\leq \epsilon$ implies
\begin{equation}
\begin{split}
(X_2-X_1)'&=(X_2-X_1)(G-1)+R_2-R_1\\
&\geq (X_2-X_1)(G-1)+\delta\\
&\geq -\epsilon|G-1|+\delta\\
&>0
\end{split}
\end{equation}

Hence $X_2-X_1$ stays positive and does not tend to zero along $\Theta$. We reach a contradiction. The limit for $Y_1$ must be $0$. 

Note that $\sqrt{2m+1}Y_1-X_1$ is positive initially along each $\Theta$. Suppose $\sqrt{2m+1}Y_1-X_1=0$ for the first time at some $\eta_*$, then at $\Theta(\eta_*)$ we have
\begin{equation}
\label{derivative of X2-X1}
\begin{split}
(\sqrt{2m+1}Y_1-X_1)'(\Theta(\eta_*))&=\sqrt{2m+1}Y_1(X_1-X_2-G+1)-R_1
\end{split}
\end{equation}
By the identity $X_1+2X_2+4mX_3=1$, we have 
\begin{equation}
\label{upper bound for G}
\begin{split}
G&=X_1^2+2X_2^2+4mX_3^2\\
&=\frac{2m+1}{2m}X_2\left(2X_2-\frac{2}{2m+1}(1-X_1)\right)+\frac{4m+1}{4m}X_1^2-\frac{X_1}{2m}+\frac{1}{4m}\\
\end{split}.
\end{equation}
Since $1-X_1-2X_2=4mX_3\geq 0$ by Proposition \ref{nonnegative X}, the first term of the computation above is no larger than $\frac{2m+1}{2m}\left(
\frac{1-X_1}{2}\right)\left(1-X_1-\frac{2}{2m+1}(1-X_1)\right)$ for any fixed $X_1$. Hence we have 
\begin{equation}
\begin{split}
G&\leq \frac{3}{2}X_1^2-X_1+\frac{1}{2}
\end{split}
\end{equation}
by replacing $X_2$ with $\frac{1-X_1}{2}$ in \eqref{upper bound for G}. As $X_2\geq X_1$ in $\mathcal{S}$, it is clear that $X_1\in\left[0,\frac{1}{3}\right]$. Hence we know that $G\leq \frac{1}{2}$ at $\Theta(\eta_*)$. Then \eqref{derivative of X2-X1} continues as
\begin{equation}
\begin{split}
(\sqrt{2m+1}Y_1-X_1)'(\Theta(\eta_*))&\geq \sqrt{2m+1}Y_1\left(X_1-X_2+\frac{1}{2}\right)-R_1\\
&= \sqrt{2m+1}Y_1\left(\frac{3}{2}X_1+2mX_3\right)-R_1\\
&\geq \frac{3}{2}\sqrt{2m+1}Y_1X_1-R_1\quad \text{as $X_3\geq 0$ in $\mathcal{S}$}\\
&= Y_1^2\left(\frac{3(2m+1)}{2}-\frac{1}{2}Y_2^2-\frac{m}{4}Y_3^2\right)\quad \text{on replacing $X_1$ with $\sqrt{2m+1}Y_1$}\\
&\geq Y_1^2\left(\frac{3(2m+1)}{2}-\frac{2m+1}{2}Y_2^2\right)\quad \text{since $2Y_2-Y_3\geq 0$}\\
&\geq 0 \quad \text{by \eqref{upper bound Y2}}
\end{split}.
\end{equation}
Hence $\sqrt{2m+1}Y_1-X_1\geq 0$ along $\Theta$. As $\lim\limits_{\eta\to\infty}Y_1(\Theta(\eta))=0$, we must have 
$\lim\limits_{\eta\to\infty}X_1(\Theta(\eta))=0$.
\end{proof}

\begin{remark}
\label{bohm functional}
The B\"ohm functional introduced in \cite{bohm_non-compact_1999} becomes $\frac{Y_2^{2m+3}Y_3^{2m}}{Y_1}$ and it is clear that
$$
\left(\frac{Y_2^{2m+3}Y_3^{2m}}{Y_1}\right)'=\frac{Y_2^{2m+3}Y_3^{2m}}{Y_1}\left((4m+3)G-1\right)\geq 0.
$$
Since $Y_1$ converges to $0$, the functional blow up at infinity instead of converging to a finite number. This brings up a difficulty in describing the $\omega$-limit set, which does not occur in two-summand case. One may consider the B\"ohm functional $Y_2^{2m+2}Y_3^{2m}$ for the two-summand type subsystem on $\mathcal{B}_{ALC}$. However, the functional only demonstrate monotonicity in the subsystem. 
\end{remark}

Asymptotic limit for integral curves of two-summand type are known \cite{wink_cohomogeneity_2017}\cite{chi_cohomogeneity_2019}. For  $\mathcal{B}_{Rd}$, we know that $\lim\limits_{\eta\to\infty}\gamma_{(1,0,0)}=\lim\limits_{\eta\to\infty}\zeta_{(1,0,0)}=P_{AC-2}$. As for $\mathcal{B}_{FS}$, we have the following.
\begin{lemma}
For all $m\geq 1$, we have
$\lim\limits_{\eta\to \infty}\gamma_{(0,1,0)}=P_{ALC-1}$.
\end{lemma}
\begin{proof}
The integral curve $\gamma_{(0,1,0)}$ lies in $\mathcal{B}_{FS}$, where $X_2\equiv X_3$ and $2Y_2\equiv Y_3$. By the definition of $\mathcal{S}$, we know that 
$$
(4m+3)X_2\geq X_1+(4m+2)X_2=1.
$$
Combining Proposition \ref{nonnegative X}, we know that $X_2\in \left[\frac{1}{4m+3},\frac{1}{4m+2}\right]$ along $\gamma_{(0,1,0)}$. Along the integral curve we have
\begin{equation}
\begin{split}
Y_2'&=Y_2(G-X_2)\\
&=Y_2((4m+2)X_2-1)((4m+3)X_2-1)\\
&\quad\text{on replacing $X_1$ with $1-(4m+2)X_2$ and $X_3$ with $X_2$}\\
&\leq 0
\end{split}.
\end{equation}
Hence $Y_2$ converges along $\gamma_{(0,1,0)}$.
Since we know that $X_1$ and $Y_1$ converge to $0$ along $\gamma_{(0,1,0)}$ by Proposition \ref{X1Y1 vanish}, we learn that $\lim\limits_{\eta\to \infty} X_2\left(\gamma_{{(0,1,0)}}(\eta)\right)=\frac{1}{4m+2}$. Hence the limit must be $P_{ALC-1}$.
\end{proof}

In order to study the asymptotics of the other integral curves of Ricci-flat metrics, we need the following propositions.
\begin{proposition}
\label{very technical}
Let $\Theta$ be any of $\gamma_{(s_1,s_2,0)}$ with $s_2> 0$, $\zeta_{(s_1,s_2,0)}$ with $s_2>0$ or $\Gamma_0$ in Theorem \ref{zeta curve}-\ref{Gamma curve}. There exists a neighborhood $U$ around $P_{ALC-1}$ such that $\big(\frac{\sqrt{2}}{2}(2Y_2-Y_3)+X_3-X_2\big)'(\Theta(\eta))>0$ as long as $\Theta(\eta)\in U\cap \{X_2-X_3>0\}$.
\end{proposition}
\begin{proof}
Fix any $\eta\in\mathbb{R}$. Let $\epsilon_1=(X_2-X_3)(\Theta(\eta))$ and $\epsilon_2=(2Y_2-Y_3)(\Theta(\eta))$. We know that $\epsilon_2$ and $\frac{\sqrt{2}}{2}\epsilon_2-\epsilon_1$ are positive since $\Theta$ is in $\mathcal{S}$. Note that
\begin{equation}
\begin{split}
&\left(\frac{\sqrt{2}}{2}(2Y_2-Y_3)+X_3-X_2\right)'\\
&=\left(\frac{\sqrt{2}}{2}(2Y_2-Y_3)+X_3-X_2\right)G-(X_3-X_2)-\sqrt{2}Y_2X_2-\frac{\sqrt{2}}{2}Y_3(X_2-2X_3)\\
&\quad +(2Y_2-Y_3)\left(\frac{1}{8}Y_1^2\left(2Y_2+Y_3\right)+\frac{m+1}{2}Y_3-Y_2\right)\\
&=\left(\frac{\sqrt{2}}{2}(2Y_2-Y_3)+X_3-X_2\right)G-(X_3-X_2)-\frac{\sqrt{2}}{2}(2Y_2-Y_3)X_3-\frac{\sqrt{2}}{2}(2Y_2+Y_3)(X_2-X_3)\\
&\quad +(2Y_2-Y_3)\left(\frac{1}{8}Y_1^2\left(2Y_2+Y_3\right)+\frac{m+1}{2}Y_3-Y_2\right)\\
&\geq  \left(\frac{\sqrt{2}}{2}\epsilon_2-\epsilon_1\right) G+\epsilon_1-\frac{\sqrt{2}}{2}\epsilon_2 X_3-\frac{\sqrt{2}}{2}(2Y_2+Y_3)\epsilon_1+\epsilon_2\left(\frac{m+1}{2}Y_3-Y_2\right)\\
&\geq \epsilon_1\left(1-\frac{\sqrt{2}}{2}(2Y_2+Y_3)\right)+\epsilon_2\left(\frac{m+1}{2}Y_3-Y_2-\frac{\sqrt{2}}{8m}\right)
\end{split}.
\end{equation}
It is straightforward to check that coefficients of $\epsilon_1$ and $\epsilon_2$ above are positive at $P_{ALC-1}$. Hence we can find a neighborhood $U$ around $P_{ALC-1}$ in which coefficients of $\epsilon_1$ and $\epsilon_2$ above are positive. If $\Theta(\eta_*)\in U\cap \{X_2-X_3>0\}$, then we see that $\left(\frac{\sqrt{2}}{2}(2Y_2-Y_3)+X_3-X_2\right)'(\Theta(\eta_*))$ must be positive.
\end{proof}

\begin{lemma}
Let $\Theta$ be any of $\gamma_{(s_1,s_2,0)}$ with $s_2>0$, $\zeta_{(s_1,s_2,0)}$ with $s_2>0$ and $\Gamma_0$ in Theorem \ref{zeta curve}-\ref{Gamma curve}, we have $\lim\limits_{\eta\to\infty}\Theta(\eta)=P_{ALC-2}$ 
\end{lemma}
\begin{proof}
Suppose the function $X_3-X_2$ vanishes finitely many times along $\Theta$. Then it eventually has a sign. Since $\left(\frac{Y_2}{Y_3}\right)'=2\frac{Y_2}{Y_3}(X_3-X_2)$, the function $\frac{Y_2}{Y_3}$ eventually monotonic decreases or increases. Hence $\lim\limits_{\eta\to\infty}\frac{Y_2}{Y_3}(\Theta(\eta))=l$ for some $l$. If $l=0$, then we must have $\lim\limits_{\eta\to\infty}Y_2(\Theta(\eta))=\lim\limits_{\eta\to\infty}Y_3(\Theta(\eta))=0$. By Proposition \ref{X1Y1 vanish}, we conclude that $\lim\limits_{\eta\to\infty}(\Theta(\eta))=(0,a,b,0,0,0)$, where $2a+4mb=2a^2+4mb^2=1$. But then one of $a$ and $b$ must be negative, a contradiction to Proposition \ref{nonnegative X}. Hence we must have $l>0$. 
Then we learn that the $\omega$-limit set of $\Theta$ contains some element in $\left\{\left(0,\frac{1}{4m+2},\frac{1}{4m+2},0,y_2,y_3\right)\mid \frac{y_2}{y_3}=l\right\}\cap \partial \mathcal{E}=\{P_{ALC-1}, P_{ALC-2}\}$. Suppose $P_{ALC-1}$ were in the $\omega$-limit set. Then $\frac{Y_2}{Y_3}$ converges to $\frac{1}{2}$. Since $\frac{1}{2}$ is the minimum value for $\frac{Y_2}{Y_3}$ in $\mathcal{S}$ and $X_3-X_2$ is assumed to have a sign eventually, we know that $X_3-X_2$ must be negative eventually. Consider
\begin{equation}
\begin{split}
(X_3-X_2)'&=(X_3-X_2)(G-1)+R_3-R_2\\
&=-(X_3-X_2)(R_1+2R_2+4mR_3)\\
&\quad +(2Y_2-Y_3)\left(\frac{1}{8}Y_1^2\left(2Y_2+Y_3\right)+\frac{m+1}{2}Y_3-Y_2\right)
\end{split}.
\end{equation}
Since $\frac{Y_2}{Y_3}$ tends to $\frac{1}{2}$ and it is clear that $R_s=R_1+2R_2+4mR_3\geq 0$ in $\mathcal{S}$, $(X_3-X_2)'$ is eventually positive along $\Theta$. Hence $X_3-X_2$ eventually monotonic increases. Then we conclude that $\Theta$ has to converge to $P_{ALC-1}$. But that implies $\Theta$ eventually enters the set $U\cap\{X_2-X_3>0\}$ constructed in Proposition \ref{very technical} and does not come out, which means that the function $\frac{\sqrt{2}}{2}(2Y_2-Y_3)+X_3-X_2$ cannot converges to zero along $\Theta$. Hence we reach a contradiction. Therefore, $P_{ALC-2}$ is in the $\omega$-limit set of $\Theta$. Since the point is a sink in $\mathcal{B}_{RF}$,  we have $\lim\limits_{\eta\to\infty}\Theta(\eta)=P_{ALC-2}$.

Suppose the function $X_3-X_2$ vanishes infinitely many times along $\Theta$. Then it is necessary that the function $R_3-R_2$ changes sign infinitely many times along $\Theta$. But
$$
R_3-R_2=(2Y_2-Y_3)\left(\frac{1}{8}Y_1^2\left(2Y_2+Y_3\right)+\frac{m+1}{2}Y_3-Y_2\right)
$$
and $Y_1$ converges to $0$ by Proposition \ref{X1Y1 vanish}.
Hence there exists a sequence $\{\eta_k\}_{k=1}^\infty$ with $\lim\limits_{k\to \infty}\eta_k=\infty$ such that $\lim\limits_{k\to \infty}(2Y_2-(m+1)Y_3)(\eta_k)=0$ and $(X_3-X_2)(\eta_k)\geq 0$ for each $k$. Therefore, combining Proposition \ref{nonnegative X}, the $\omega$-limit set of $\Theta$ must contain some point $P_*$ in the set 
$$\{(0,x_2,x_3,0,y_2,y_3)\mid x_2,x_3\geq 0, 2x_2+4mx_3=1, 2y_2-(m+1)y_3=0\}\cap \partial \mathcal{E}.$$
If $P_*=P_{ALC-2}$, then $\Theta$ converges to $P_{ALC-2}$ since the point is a sink in the subsystem restricted on $\mathcal{B}_{RF}$. Suppose $P_*\neq P_{ALC-2}$, then it is not a critical point. Since $\mathcal{B}_{ALC}$ is a $2$-dimensional invariant set and the $\omega$-limit set is flow-invariant, the $\omega$-limit set of $\Theta$ must contain the integral curve $\widetilde{\Theta}$ that contains $P_*$ and lies on $\mathcal{B}_{ALC}$. Note that the reduced system on $\mathcal{B}_{ALC}$ is essentially the two-summand type cohomogeneity one system. 
Based on the study in \cite{wink_cohomogeneity_2017}\cite{chi_cohomogeneity_2019}, we know that $\widetilde{\Theta}$ must converges to $P_{ALC-2}$. Specifically, recall Remark \ref{bohm functional} and consider the B\"ohm functional $Y_2^{2m+2}Y_3^{2m}$. We have
$$
(Y_2^{2m+2}Y_3^{2m})'=Y_2^{2m+2}Y_3^{2m}((4m+2)G-1)\geq 0
$$
when restricted on $\mathcal{B}_{ALC}$. Hence $Y_2^{2m+2}Y_3^{2m}$ increases monotonically to some positive number along $\widetilde{\Theta}$, and the $\omega$-limit set of $\widetilde{\Theta}$ contains some element in $\{P_{ALC-1},P_{ALC-2}\}$. Since $P_{ALC-1}$ is in the boundary of the 2-dimensional invariant set $\mathcal{S}\cap \mathcal{B}_{ALC}$ while $P_{ALC-2}$ is in the interior, one can exclude $P_{ALC-1}$ by perturbing the boundary of $\mathcal{S}\cap \mathcal{B}_{ALC}$. Hence $\widetilde{\Theta}$ converges to $P_{ALC-2}$ and therefore $P_{ALC-2}$ is in the $\omega$-limit set of $\Theta$. The proof is complete.
\end{proof}

The asymptotic limits of all integral curves that represent Ricci-flat metrics are known, as summarized in the following lemma.
\begin{lemma}
\label{Asym for RF}
Asymptotic limits of integral curves in Lemma \ref{completeness} are the following.
\begin{equation}
\begin{split}
&\lim\limits_{\eta\to\infty} \zeta_{(s_1,s_2,0)}=\left\{\begin{array}{ll}
P_{AC-2}& s_2=0\\
P_{ALC-2}& s_1,s_2>0\\
\end{array}\right., \quad \lim\limits_{\eta\to\infty} \gamma_{(s_1,s_2,0)}=\left\{\begin{array}{ll}
P_{AC-2}& s_2=0\\
P_{ALC-2}& s_1,s_2>0\\
P_{ALC-1}& s_1=0\\
\end{array}\right., \\
&\lim\limits_{\eta\to\infty} \Gamma_0=P_{ALC-2}.
\end{split}
\end{equation}

\end{lemma}

\subsection{Asymptotics for Negative Einstein metrics}
\begin{proposition}
\label{when in PAH}
Points in $\mathcal{S}$ with $G+\Lambda W^2=0$ must lie in the 1-dimensional stable manifold $P_{AH}$.
\end{proposition}
\begin{proof}
By the definition of the function $W$, we have $1-G-R_s=-(4m+2)\Lambda W^2$ in $\mathcal{E}$. Since $X_1+2X_2+4mX_3=1$ is held, we obtain the lower bound for $G\geq \frac{1}{4m+3}$ using Cauchy--Schwarz inequality. We have
\begin{equation}
\label{range of w}
\begin{split}
1-\frac{1}{4m+3}&\geq 1-G \\
&\geq -(4m+2)\Lambda W^2 \quad \text{since $R_s\geq 0$ in $\mathcal{S}$}
\end{split}.
\end{equation}
 Hence $-\Lambda W^2\leq \frac{1}{4m+3}$ in $\mathcal{S}.$
But by the assumption on the point, we have $0=G+\Lambda W^2 \geq \frac{1}{4m+3}+\Lambda W^2$. Hence we are forced to have $-\Lambda W^2=\frac{1}{4m+3}$ and $G=\frac{1}{4m+3}$. Then $R_s$ is forced to vanish at such a point. The point must lie in $P_{AH}$.
\end{proof}

\begin{lemma}
\label{Asym for negative Einstein}
Let $\Theta$ be any of integral curves $\zeta_{(s_1,s_2,s_3)}$ with $s_3>0$, $\gamma_{(s_1,s_2,s_3)}$ with $s_3>0$ or $\Gamma_s$ in Lemma \ref{completeness} with $s>0$. We have $\lim\limits_{\eta\to\infty} \Theta=P_{AH}(y_1)$ for some $y_1\in[0,\sqrt{2}]$.
\end{lemma}
\begin{proof}
Since these integral curves are trapped in $\mathcal{S}$, we have 
$
1-\frac{1}{4m+3}\geq -(4m+2)\Lambda W^2
$
as in \eqref{range of w}. Then $W'=W(G+\Lambda W^2)\geq 0$. Hence the function $W$ is increasing along $\Theta$ and converges to some positive number. Then there exists a sequence $\{\eta_k\}_{k=1}^\infty$ with $\lim\limits_{k\to\infty}\eta_k=\infty$ such that $\lim_{k\to \infty}(G+\Lambda W^2)(\Theta(\eta_k))=0$. Therefore, some subset of $P_{AH}$ is in the $\omega$-limit set of these integral curves by Proposition \ref{when in PAH}. The proof is complete by Lemma \ref{Critical in W neq 0}.
\end{proof}

For $\zeta_{(s_1,0,s_3)}$ and $\gamma_{(s_1,0,s_3)}$, we know that they converge to $P_{AH}(\sqrt{2})$. We are yet to determine what point in $P_{AH}$ that $\zeta_{(s_1,s_2,s_3)}$ and $\gamma_{(s_1,s_2,s_3)}$ converges to if $s_2>0$. Note that although $Y_1$ decreases in this case, it does not necessarily need to converge to zero.

\section{Relation to Special Holonomy}
In this section, we check the holonomy of Einstein metrics in Theorem \ref{zeta curve}-\ref{Gamma curve}. Some known results are recovered. 

\label{Holonomy}
\subsection{Negative K\"ahler--Einstein and Calabi--Yau}
\label{Calabi--Yau}
We recover K\"ahler--Einstein metrics with a complex structure $\mathcal{I}$ in \cite{dancer_kahler-einstein_1998} that is preserved by the action of $G$. Recall Remark \ref{group L} that $L=Sp(m)U(1)U(1)$.
If $dt^2+g_{G/K}(t)$ is K\"ahler--Einstein, then the coadjoint orbit $G/L=\mathbb{CP}^{2m+1}$ is K\"ahler for each $t$. Consequently, the cohomogeneity one K\"ahler--Einstein condition boils down to
\begin{equation}
\label{Kahler condition}
\begin{split}
&c\dot{c}=\frac{a}{4}\\
&2c^2=b^2
\end{split}.
\end{equation}
The second equation above is equivalent to the coadjoint orbit $G/L$ being K\"ahler. In the new coordinate with variables defined in \eqref{new variable}, integral curves that represent K\"ahler--Einstein metrics must lie in 
$$
\mathcal{B}_{KE}:=\mathcal{B}_{FS}\cap \left\{X_3\equiv \frac{1}{4}Y_1Y_3\right\}.
$$
We check the following.
\begin{proposition}
The set $\mathcal{B}_{KE}$ is invariant.
\end{proposition}
\begin{proof}
It is clear that $\mathcal{B}_{FS}$ is invariant. If $X_3=\frac{1}{4}Y_1Y_3$ in $\mathcal{B}_{FS}$, then $X_1=1-2X_2-4mX_3=1-(4m+2)X_3=1-\frac{2m+1}{2}Y_1Y_3$ in $\mathcal{B}_{FS}$. Hence on $\mathcal{B}_{KE}$, we can eliminate all $X_i$'s and $Y_2$ in \eqref{Polynomial Conservation Law} and obtain the following.
\begin{equation}
\label{KE Polynomial Conservation Law}
\begin{split}
\frac{m+1}{2}Y_3^2+\frac{2m+1}{8}Y_1^2Y_3^2-\frac{1}{2}Y_1Y_3-\Lambda W^2=0
\end{split}
\end{equation}
On the other hand, we have
\begin{equation}
\begin{split}
&\left.\left\langle\nabla\left(X_3-\frac{1}{4}Y_1Y_3\right),V_{\leq 0}\right\rangle\right|_{X_3-\frac{1}{4}Y_1Y_3=0}\\
&=\left(X_3-\frac{1}{4}Y_1Y_3\right)(G+\Lambda W^2-1)\\
&\quad + (m+2)Y_2Y_3-\frac{1}{8}Y_1^2Y_3^2-\frac{1}{2}Y_3^2-\Lambda W^2-\frac{1}{4}Y_1Y_3(1+X_1-2X_3)\\
&= \frac{m+2}{2}Y_3^2-\frac{1}{8}Y_1^2Y_3^2-\frac{1}{2}Y_3^2-\Lambda W^2-\frac{1}{4}Y_1Y_3(2-(m+1)Y_1Y_3)\\
&\quad \text{Use definition of $\mathcal{B}_{KE}$ to eliminate $Y_2$ and $X_i$'s}\\
&=\frac{m+1}{2}Y_3^2+\frac{2m+1}{8}Y_1^2Y_3^2-\frac{1}{2}Y_1Y_3-\Lambda W^2\\
&=0 \quad \text{by \eqref{KE Polynomial Conservation Law}}
\end{split}.
\end{equation}
Hence $\mathcal{B}_{KE}$ is invariant.
\end{proof}
Hence $\mathcal{B}_{KE}$ is an 2-dimensional invariant set. It straightforward to check that $\mathcal{B}_{KE}$ only contains critical points $P_{AC-1}$, $(1,0,0,0,0,0)$ and $\left(-\frac{4m+1}{4m+3},\frac{2}{4m+3},\frac{2}{4m+3},0,0,0\right)$ listed in Section \ref{Critical Points}. The last two critical points are of type 7 in Section \ref{Critical Points}. Since $\mathcal{B}_{KE}$  does not contain $P_0$, $P_{AC-2}$, $P_{ALC-1}$, $P_{ALC-2}$ or any point on $P_{AH}$, no integral curve of Theorem \ref{zeta curve}-\ref{Gamma curve} lies in $\mathcal{B}_{KE}$.

One can check that there are integral curves emanating from $(1,0,0,0,0,0)$. They represent K\"ahler--Einstein metrics constructed in \cite{berard-bergery_sur_1982}\cite[Theorem 9.129]{besse_einstein_2008}. In particular, $\mathcal{B}_{KE}\cap \mathcal{B}_{RF}$ is a 1-dimensional invariant set that contains $P_{AC-1}$ and $(1,0,0,0,0,0)$. The part that ``joins'' these two critical points is exactly the 
image of the integral curve that emanates from $(1,0,0,0,0,0)$ and tends to $P_{AC-1}$, representing a Calabi--Yau metric with a $\mathbb{CP}^{2m+1}$ bolt and an AE limit.

\subsection{Quaternionic K\"ahler and Hyper-K\"ahler}
\label{QK and HK}
By \cite{dancer_quaternionic_1999}, the existence of the triple of almost complex structures forces $a$ and $b$ to be linear function in $t$ and $\frac{a}{b}=\sqrt{2}$. Therefore, any integral curve that represents a hyperK\"ahler metric or a quaternionic K\"ahler metric must lie in the invariant set $\mathcal{B}_{Rd}$. For a quaternionic K\"ahler metric with normalized Einstein constant $\Lambda=-(4m+3)$,  the closedness of the fundamental 4-form implies
\begin{equation}
\label{Tricomplex condition}
\begin{split}
&c\dot{c}=\frac{a}{4}\\
&2c^2=b^2+\frac{2}{m+3}\Lambda W^2
\end{split}.
\end{equation}
Therefore, integral curves that represent quaternionic K\"ahler metrics must lie in the following set.
$$\mathcal{B}_{QK}:=\mathcal{B}_{Rd}\cap \left\{Y_3^2-2Y_2Y_3+\frac{2}{m+3}\Lambda W^2\equiv 0\right\} \cap \left\{X_3-\frac{1}{4}Y_1Y_3\equiv 0\right\}.$$
\begin{proposition}
The set $\mathcal{B}_{QK}$ is invariant.
\end{proposition}
\begin{proof}
It is clear that $\mathcal{B}_{Rd}$ is invariant. Moreover, $X_3=\frac{1}{4}Y_1Y_3$ becomes $X_3=\frac{\sqrt{2}}{4}Y_3$ in $\mathcal{B}_{Rd}$ and $X_1=X_2=\frac{1-4mX_3}{3}=\frac{1-m\sqrt{2}Y_3}{3}$ in $\mathcal{B}_{Rd}$. Hence on $\mathcal{B}_{QK}$, we can eliminate $Y_1$, $W$ and all $X_i$'s in \eqref{Polynomial Conservation Law} and obtain the following.
\begin{equation}
\label{QK Polynomial Conservation Law}
\begin{split}
0=\left(1-\frac{2m+3}{2}\sqrt{2}Y_3+\frac{3}{2}\sqrt{2}Y_2\right)\left(1+\frac{4m+3}{2 }\sqrt{2}Y_3-\frac{3}{2}\sqrt{2}Y_2\right)
\end{split}
\end{equation}
Note that by the definition of $\mathcal{B}_{QK}$, we must have $Y_3\geq 2Y_2$. Hence computation above implies
$$1-\frac{2m+3}{2}\sqrt{2}Y_3+\frac{3}{2}\sqrt{2}Y_2=0$$ on $\mathcal{B}_{QK}$.

On the other hand, we have
\begin{equation}
\begin{split}
&\left.\left\langle\nabla\left(X_3-\frac{1}{4}Y_1Y_3\right),V_{\leq 0}\right\rangle\right|_{X_3-\frac{1}{4}Y_1Y_3=0}\\
&=\left(X_3-\frac{1}{4}Y_1Y_3\right)(G+\Lambda W^2-1)\\
&\quad + (m+2)Y_2Y_3-\frac{1}{8}Y_1^2Y_3^2-\frac{1}{2}Y_3^2-\Lambda W^2-\frac{1}{4}Y_1Y_3(1+X_1-2X_3)\\
&=(m+2)Y_2Y_3-\frac{3}{4}Y_3^2+\frac{m+3}{2}(Y_3^2-2Y_2Y_3)\\
&\quad -\frac{\sqrt{2}}{4}Y_3\left(\frac{4}{3}-\frac{2m+3}{6}\sqrt{2}Y_3\right)\\
&\quad \text{Use definition of $\mathcal{B}_{Rd}$ to eliminate $Y_1$, $W$ and $X_i$'s}\\
&=\frac{\sqrt{2}}{3}Y_3\left(\frac{2m+3}{2}\sqrt{2}Y_3-\frac{3}{2}\sqrt{2}Y_2-1\right)\\
&=0 \quad \text{by \eqref{QK Polynomial Conservation Law}}
\end{split}
\end{equation}
and
\begin{equation}
\begin{split}
&\left.\left\langle\nabla\left(Y_3^2-2Y_2Y_3+\frac{2}{m+3}\Lambda W^2\right),V_{\leq 0}\right\rangle\right|_{Y_3^2-2Y_2Y_3+\frac{2}{m+3}\Lambda W^2=0}\\
&=2\left(Y_3^2-2Y_2Y_3+\frac{2}{m+3}\Lambda W^2\right)(G+\Lambda W^2) + Y_3^2(2X_2-4X_3)+4Y_2Y_3X_3\\
&=\frac{2}{3}Y_3^2\left(1-\frac{2m+3}{2}\sqrt{2}Y_3+\frac{3}{2}\sqrt{2}Y_2\right)\\
&\quad \text{Use definition of $\mathcal{B}_{Rd}$ to eliminate $X_i$'s}\\
&=0 \quad \text{by \eqref{QK Polynomial Conservation Law}}
\end{split}.
\end{equation}
Therefore the proof is complete.
\end{proof}

Critical points $P_{AC-1}$ and $P_{QK}$ are in the set $\mathcal{B}_{QK}$ and the set is 1-dimensional. The quaternionic K\"ahler metric in \cite{swann_hyper-kahler_1991} is realized as the integral curve $\gamma_{\left(-\frac{1}{\sqrt{(4m+12)^2+1}},0,\frac{4m+12}{\sqrt{(4m+12)^2+1}}\right)}$. At the infinity, the exponential index for $a$ and $b$ is twice the one of $c$. As $Y_3\geq 2Y_2$ in $\mathcal{B}_{QK}$, we know that such an integral curve is not contained in $\mathcal{S}$ hence it is not any one of the metrics in Theorem \ref{zeta curve}-\ref{Gamma curve}. Note that the hyper-K\"ahler metric is represented by the critical point $P_{AC-1}$, which is the flat metric $\gamma_{(0,0,0)}$ on $\mathbb{R}^{4m+4}$.

\subsection{$\spin(7)$}
\label{m=1}
In the case $m=1$, it is known that there exists $\spin(7)$ metrics on $M^8$ and $\mathbb{R}^8$\cite{cvetic_new_2004}. From \cite{hitchin_stable_2001}\cite{cvetic_new_2004}, we can write down the $\spin(7)$ condition.
\begin{equation}
\label{old spin(7)}
\begin{split}
&\frac{\dot{a}}{a}=\frac{1}{2}\frac{a}{b^2}-\frac{1}{2}\frac{a}{c^2}\\
&\frac{\dot{b}}{b}=\sqrt{2}\frac{1}{b}-\frac{\sqrt{2}}{2}\frac{b}{c^2}-\frac{1}{2}\frac{a}{b^2}\\
&\frac{\dot{c}}{c}=\frac{\sqrt{2}}{2}\frac{b}{c^2}+\frac{1}{4}\frac{a}{c^2}
\end{split}.
\end{equation}

Define
\begin{equation}
\begin{split}
F_1&=X_1-\frac{1}{2}Y_1Y_2+\frac{1}{2}Y_1Y_3\\
F_2&=X_2-\sqrt{2}Y_2+\frac{\sqrt{2}}{2}Y_3+\frac{1}{2}Y_1Y_2\\
F_3&=X_3-\frac{\sqrt{2}}{2}Y_3-\frac{1}{4}Y_1Y_3
\end{split}.
\end{equation}
The $\spin(7)$ condition \eqref{old spin(7)} is transformed to $F_i=0$ in the new coordinates. Define 
$$
\mathcal{B}^-_{\spin(7)}=\mathcal{B}_{RF}\cap \{F_1\equiv F_2\equiv F_3\equiv 0\}.
$$
We can check the following.
\begin{proposition}
\label{spin set}
The set $\mathcal{B}^-_{\spin(7)}$ is invariant.
\end{proposition}
\begin{proof}
On $\mathcal{B}_{RF}$, we have
\begin{equation}
\begin{split}
&\langle\nabla F_1,V_{\leq 0}\rangle\\
&=F_1(G-1)-Y_1Y_2(F_1+2F_3)+Y_1Y_3(F_1+F_2+F_3)\\
&\langle\nabla F_2,V_{\leq 0}\rangle\\
&=F_2(G-1)-\sqrt{2}Y_2(F_1+F_2+4F_3)+\frac{\sqrt{2}}{2}Y_3(F_1+3F_2+2F_3)+Y_1Y_2(F_1+2F_3)\\
&\langle\nabla F_3,V_{\leq 0}\rangle\\
&=F_3(G-1)-\frac{\sqrt{2}}{2}Y_3(F_1+3F_2+2F_3)-\frac{1}{2}Y_1Y_3(F_1+F_2+F_3)
\end{split}.
\end{equation}
Computations show the above all vanish on $\mathcal{B}^-_{\spin(7)}$.
The proof is complete.
\end{proof}
Although the definition of $\mathbb{B}^-_{\spin(7)}$ consists of 6 equalities, one can show that $X_1+2X_2+4mX_3=1$ holds once all $F_i$'s and $1-G-R_s$ vanish. Therefore, $\mathbb{B}^-_{\spin(7)}$ is a $2$-dimensional surface and its projection to the $Y$-space is a level set given by
$$
1+\frac{1}{2}Y_1Y_2-\frac{1}{2}Y_1Y_3-2\sqrt{2}Y_2-\sqrt{2}Y_3=0.
$$
By changing the sign of $a$. we obtain the $\spin(7)$ condition with the opposite chirality. 
\begin{equation}
\begin{split}
H_1&=X_1+\frac{1}{2}Y_1Y_2-\frac{1}{2}Y_1Y_3\\
H_2&=X_2-\sqrt{2}Y_2+\frac{\sqrt{2}}{2}Y_3-\frac{1}{2}Y_1Y_2\\
H_3&=X_3-\frac{\sqrt{2}}{2}Y_3+\frac{1}{4}Y_1Y_3
\end{split}
\end{equation}
and 
$$
\mathcal{B}^+_{\spin(7)}=\mathcal{B}_{RF}\cap\{H_1\equiv H_2\equiv H_3\equiv 0\}.
$$
With the similar computation in the proof of Proposition \ref{spin set}, we can show that $\mathcal{B}^+_{\spin(7)}$ is invariant. Both invariant sets are presented in Figure \ref{spin(7)}. In our new coordinate, the $\spin(7)$ metric and the $G_2$ metric in \cite{bryant_construction_1989}\cite{gibbons_einstein_1990} are realized as straight line segments that lie in $\mathcal{B}^-_{\spin(7)}$.
\begin{figure}[h!]
\centering
\begin{subfigure}{.3\textwidth}
  \centering
  \includegraphics[width=1\linewidth]{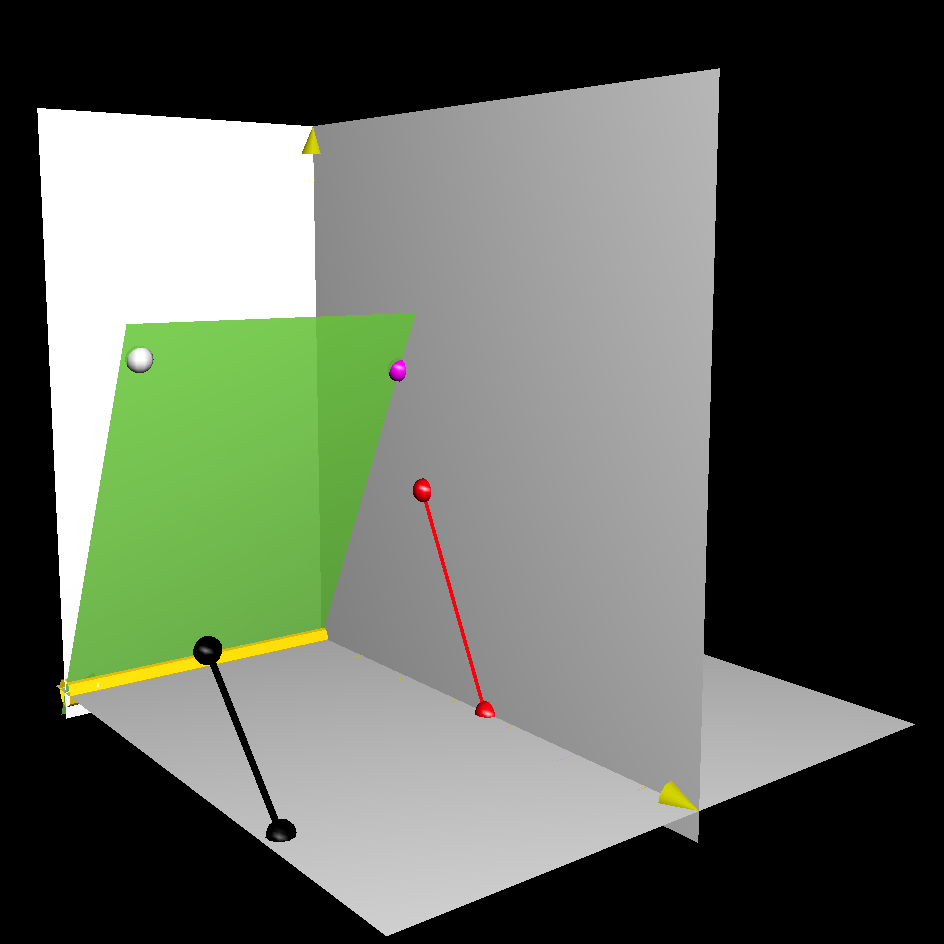}
  \caption{$\spin(7)$ and $G_2$ metrics}
  \label{fig:sfig2}
\end{subfigure}
\begin{subfigure}{.3\textwidth}
  \centering
  \includegraphics[width=1\linewidth]{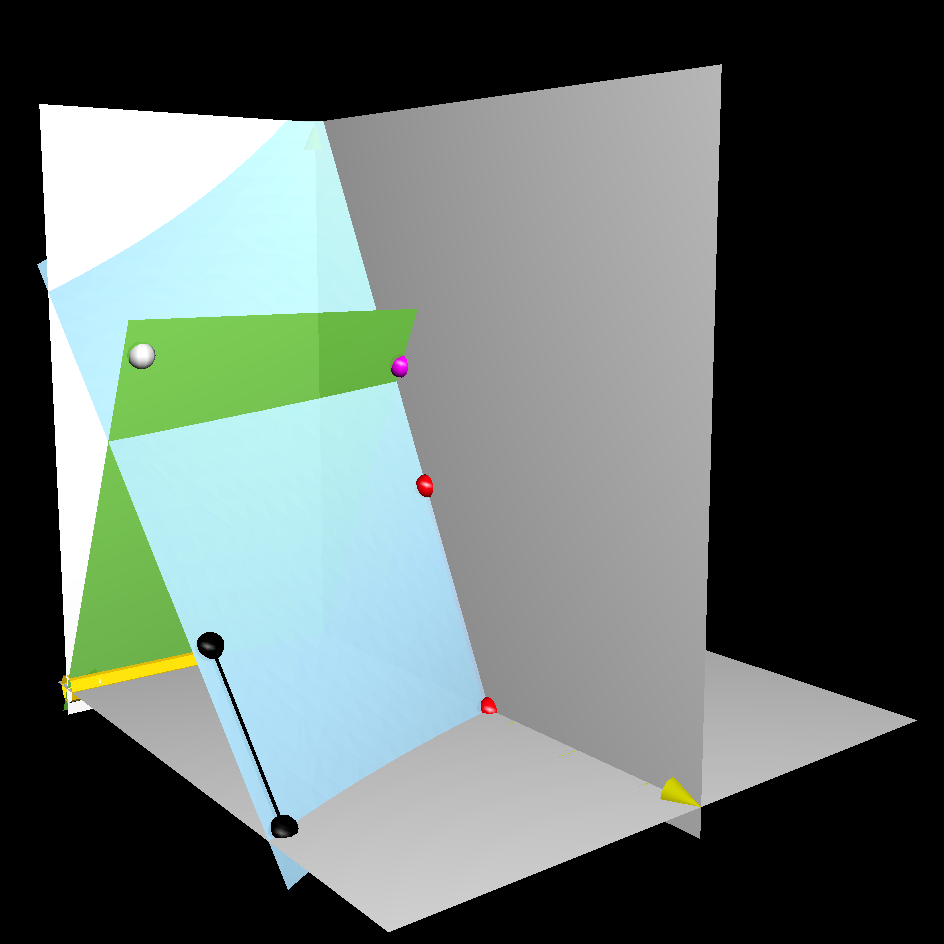}
  \caption{$\mathcal{B}^-_{\spin(7)}$}
  \label{fig:sfig2}
\end{subfigure}
\begin{subfigure}{.3\textwidth}
  \centering
  \includegraphics[width=1\linewidth]{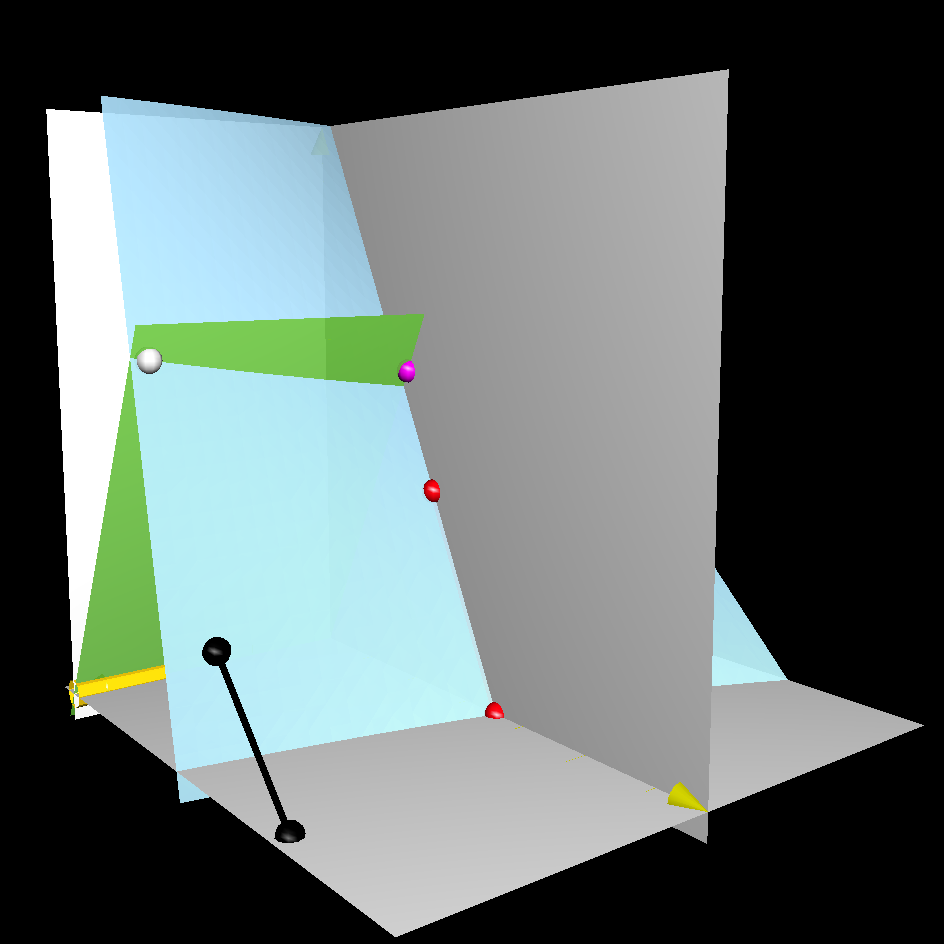}
  \caption{$\mathcal{B}^+_{\spin(7)}$}
  \label{fig:sfig2}
\end{subfigure}
\caption{Integral curves that represents $\spin(7)$ metrics (black) and $G_2$ metrics (red)}
\label{spin(7)}
\end{figure}

Linearization at $P_0$ shows that $\zeta_{(s_1,s_2,s_3)}$ lie in $\mathcal{B}^-_{\spin(7)}$ for all $(s_1,s_2,0)\in\mathbb{S}^2$ with $s_1>0$ and $s_2\geq 0$. $\zeta_{(1,0,0)}$ is the AC $\spin(7)$ metric found in \cite{bryant_construction_1989}\cite{gibbons_einstein_1990} and the 1-parameter family $\zeta_{(s_1,s_2,0)}$ with $s_2>0$ is the family of ALC $\spin(7)$ metrics found in \cite{cvetic_new_2004}. Specifically, for we obtain
$$
\zeta_{(s_1,s_2,0)}=\left\{\begin{array}{cc}
\mathbb{B}_8^+& 2s_1>s_2 \\
\mathbb{B}_8& 2s_1=s_2\\
\mathbb{B}_8^-&  2s_1<s_2
\end{array}\right.
$$
Another new $\spin(7)$ metric $\mathbb{A}_8$ was found on $\mathbb{R}^8$ in \cite{cvetic_new_2004}. This metric is locally the same as $\mathbb{B}_8$ although they differ globally. This property is reflected in our pictures as both metrics are lie in the 1-dimensional invariant set 
$$
\mathcal{B}^-_{\spin(7)}\cap\left\{\sqrt{2}Y_2-\sqrt{2}Y_3-Y_1Y_2=0\right\}.
$$
\begin{figure}[h!]
\centering
\begin{subfigure}{.3\textwidth}
  \centering
  \includegraphics[width=1\linewidth]{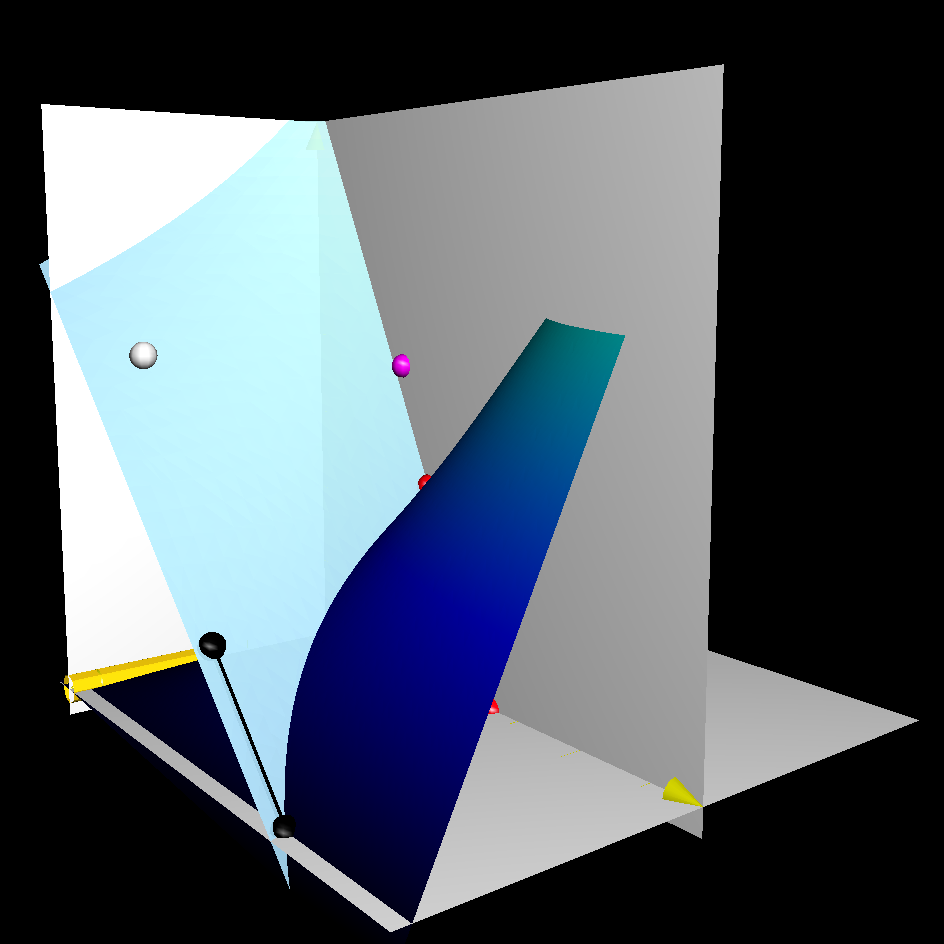}
  \caption{$\mathbb{B}_8$}
  \label{fig:sfig2}
\end{subfigure}
\begin{subfigure}{.3\textwidth}
  \centering
  \includegraphics[width=1\linewidth]{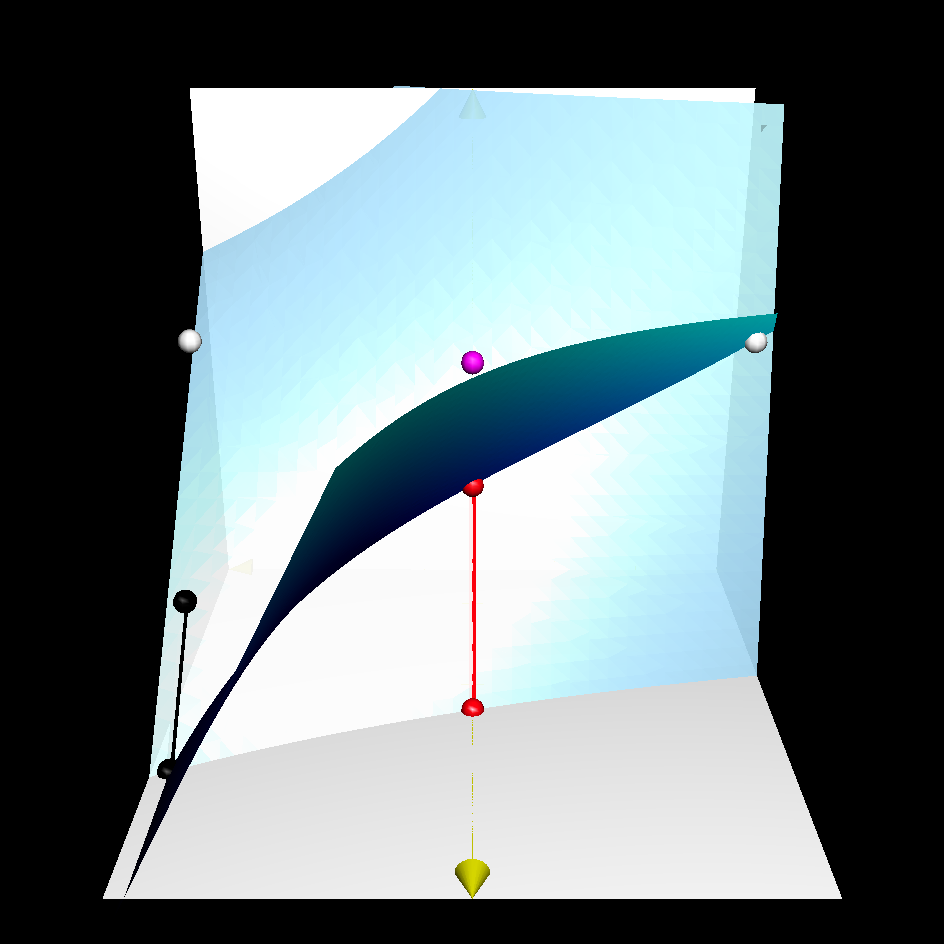}
  \caption{$\mathbb{B}_8$ and $\mathbb{A}_8$}
  \label{fig:sfig2}
\end{subfigure}
\begin{subfigure}{.3\textwidth}
  \centering
  \includegraphics[width=1\linewidth]{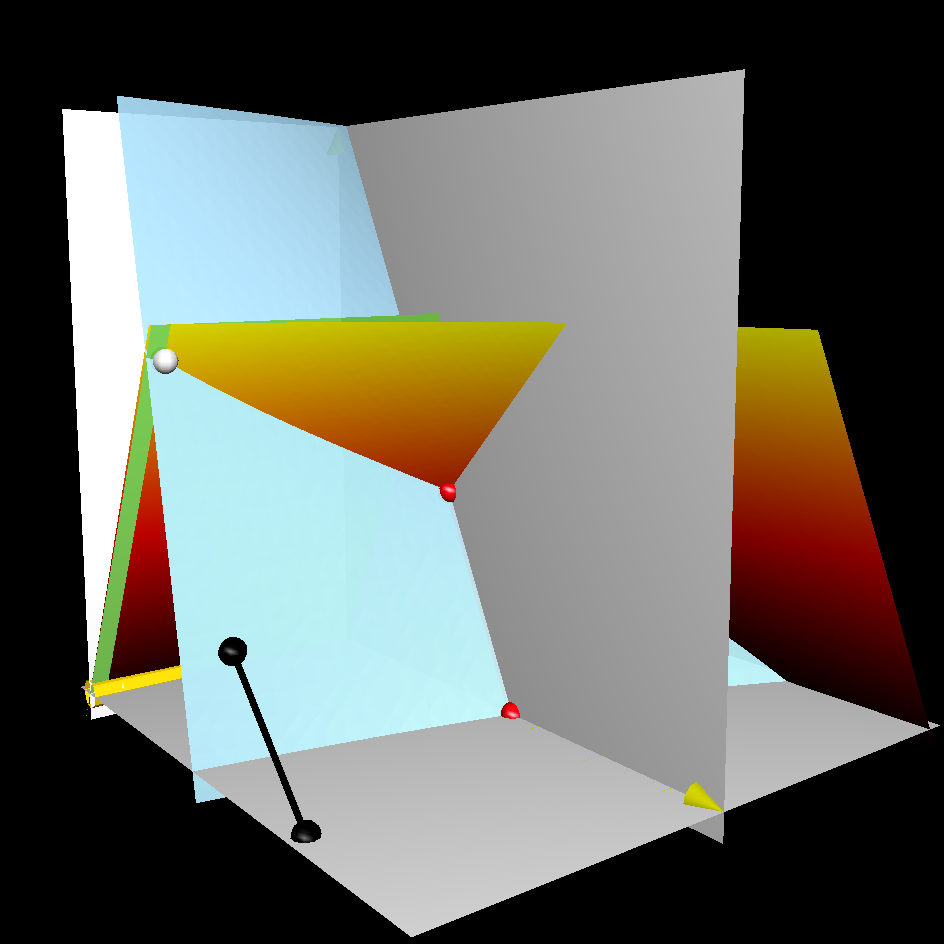}
  \caption{$\mathbb{A}_8$ with the opposite chirality}
  \label{fig:sfig2}
\end{subfigure}
\caption{$\mathbb{A}_8$ and $\mathbb{B}_8$}
\label{fig2}
\end{figure}
Simply change the sign of $Y_1$, then we can present $\mathbb{A}_8$ with the opposite chirality in the compact invariant set $\mathcal{S}$. It is realized by the integral curve $\gamma_{\left(\frac{1}{\sqrt{5}},\frac{2}{\sqrt{5}},0\right)}$. 

\begin{remark}
In \cite{cvetic_cohomogeneity_2002}, the sign change occurs in one of the $b$ component in order to obtain non-trivially different system since $\frac{a^2}{b^2}$ is not necessarily $2$. A 1-parameter family of $\spin(7)$ metric $\mathbb{C}_8$ was found in \cite{cvetic_cohomogeneity_2002}. They are metrics with Fubini--Study $\mathbb{CP}^{2m+1}$ bolt. At the infinity, one of the $b$ component tends to a constant while the other grow linearly as the same rate as $a$. Therefore, these metrics are not realized in this article as the 3-sphere $H/K$ is really controlled by three functions instead of two. However, if one further impose $2c^2=b^2$, then the metric is the Calabi--Yau metrics described in Section \ref{Calabi--Yau}.
\end{remark}

Recall in Section \ref{AC1 and AC2}, we know that there exists a unique unstable eigenvector of $\mathcal{L}(P_{AC-2})$ that is tangent to $\partial\mathcal{E}$ and $\Gamma_0$ emanates from $P_{AC-2}$ via this vector. Computation shows that this eigenvector is tangent to $\mathcal{B}^-_{\spin(7)}$. Hence $\Gamma_0$ is a singular $\spin(7)$ metric.

In general, we have the following Lemma.
\begin{lemma}
\label{generic holonomy}
Consider the case $m=1$. Metrics $\zeta_{(s_1,s_2,0)}$ and $\Gamma_0$ on $M^8$ and metrics $\gamma_{(s_1,s_2,0)}$ on $\mathbb{R}^{8}$ all have holonomy group no smaller than $\spin(7)$. In particular,
\begin{enumerate}
\item
Metrics $\zeta_{(s_1,s_2,0)}$ and $\Gamma_0$ on $M$ are $\spin(7)$.
\item
Metrics $\gamma_{\left(\frac{1}{\sqrt{5}},\frac{2}{\sqrt{5}},0\right)}$ on $\mathbb{R}^8$ is $\spin(7)$.
\item
Metrics $\gamma_{(s_1,s_2,0)}$ with
$(s_1,s_2,0)\neq \left(\frac{1}{\sqrt{5}},\frac{2}{\sqrt{5}},0\right)$ on $\mathbb{R}^8$ have generic holonomy.
\end{enumerate}
For the case $m>1$, metrics $\zeta_{(s_1,s_2,0)}$ with $s_2>0$ and $\Gamma_0$ on $M$ and metrics $\gamma_{(s_1,s_2,0)}$ with $s_2>0$ on $\mathbb{R}^{4m+4}$ have generic holonomy.
\end{lemma}
\begin{proof}
Consider the case $m=1$. By the discussion above, it is clear that metrics $\zeta_{(s_1,s_2,0)}$ and $\Gamma_0$ on $M^8$, metrics $\gamma_{\left(\frac{1}{\sqrt{5}},\frac{2}{\sqrt{5}},0\right)}$ on $\mathbb{R}^8$ are $\spin(7)$. It suffices to prove $\gamma_{(s_1,s_2,0)}$ with
$(s_1,s_2,0)\neq \left(\frac{1}{\sqrt{5}},\frac{2}{\sqrt{5}},0\right)$ on $\mathbb{R}^8$ have generic holonomy. By Lemma \ref{Asym for RF}, we know that 
$$\lim\limits_{\eta\to\infty} \gamma_{(s_1,s_2,0)}=\left\{\begin{array}{ll}
P_{AC-2}& s_2=0\\
P_{ALC-2}& s_1,s_2>0\\
P_{ALC-1}& s_1=0\\
\end{array}\right..
$$
Hence the limit space is one of the following.
\begin{enumerate}
\item
The metric cone over Jensen 7-sphere, its holonomy is $\spin(7)$. 
\item
An $\mathbb{S}^1$-bundle over the metric cone over a nearly K\"ahler $\mathbb{CP}^3$, whose holonomy group contains a subgroup $G_2$.
\item
An $\mathbb{S}^1$-bundle over the metric cone over a Fubini--Study $\mathbb{CP}^3$. The holonomy group contains a subgroup $SO(7)$. 
\end{enumerate}
Suppose the metric $\gamma_{(s_1,s_2,0)}$ admits a K\"ahler structure. By passing the K\"ahler structure to the limit space, we learn that the holonomy group of the limit space must be contained in $SU(4)$.

Note that $SU(4)$ is $15$-dimensional and simply connected. Both $\spin(7)$ and $SO(7)$ have dimension larger than 15, hence they are not contained in $SU(4)$. If the holonomy group that contains $G_2$ were also contained in $SU(4)$, then it must be $SU(4)$ itself. But if $G_2$ were contained in $SU(4)$, then $SU(4)/G_2$ must be a circle, a contradiction to the fact that $SU(4)$ is simply connected. We conclude that $G_2$ is not contained in $SU(4)$. Therefore, $\gamma_{(s_1,s_2,0)}$ with
$(s_1,s_2,0)\neq \left(\frac{1}{\sqrt{5}},\frac{2}{\sqrt{5}},0\right)$ on $\mathbb{R}^8$ must have generic holonomy.

Consider the case $m>1$. With $s_2>0$ and Lemma \ref{Asym for RF}, we have 
$$\lim\limits_{\eta\to\infty} \gamma_{(s_1,s_2,0)}=\left\{\begin{array}{ll}
P_{ALC-2}& s_1>0\\
P_{ALC-1}& s_1=0\\
\end{array}\right.,\quad \lim\limits_{\eta\to\infty} \zeta_{(s_1,s_2,0)}=\lim\limits_{\eta\to\infty} \Gamma_0=P_{ALC-2}.
$$
Then the limit space must have holonomy group that contains a subgroup $SO(4m+3)$. Since the dimension of $SO(4m+3)$ is larger than the one of $SU(2m+2)$ if $m\geq 1$. We conclude that the Ricci-flat metrics above have generic holonomy.
\end{proof}

By Lemma \ref{generic holonomy} and by Theorem 2.1 in \cite{hitchin_harmonic_1974} and \cite{wang_parallel_1989}, Theorem \ref{generic} is proven.
\\
\\
\textbf{Acknowledgements.} The author is grateful to McKenzie Wang for introducing the problem and his useful comment. The author would like to thank Cheng Yang for helpful discussions on dynamic system. The author would also like to thank Christoph B\"ohm and Lorenzo Foscolo for their helpful suggestions and remarks on this project. Lemma \ref{generic holonomy} is proven thanks to the inspiring discussion with Lorenzo Foscolo.

\bibliography{Bibliography}
\bibliographystyle{alpha}
\end{document}